\documentclass[12pt]{article}
\usepackage{fullpage}
\usepackage{amsmath,amssymb,amsthm}
\usepackage{newtxtext}
\usepackage{pifont}
\usepackage[table]{xcolor}
\usepackage{makecell}
\definecolor{forestgreen}{rgb}{0.13, 0.55, 0.13}

\usepackage{colortbl}
\usepackage[hidelinks]{hyperref}
\usepackage{url}
\hypersetup{colorlinks,linkcolor={black},citecolor={teal},urlcolor={black}}  
\usepackage[round]{natbib}
\usepackage{titlesec}
\usepackage[smallerops]{newtxmath}
\usepackage[ruled]{algorithm2e}
\usepackage{graphicx} 
\usepackage{booktabs} 
\usepackage{multirow} 
\usepackage{xcolor}
\usepackage{tcolorbox}
\tcbset{boxsep=0mm,boxrule=0pt,colframe=white,arc=0mm,left=0.5mm,right=0.5mm,opacityback=0.55 }
\usepackage{tikz}
\usepackage[textsize=footnotesize]{todonotes}
\newtheorem{theorem}{Theorem}[section]
\newtheorem{lemma}[theorem]{Lemma}

\theoremstyle{definition}
\newtheorem{definition}[theorem]{Definition}
\theoremstyle{remark}
\newtheorem*{remark}{Remark}

\newcommand{\minimize}{\mathop{\textrm{minimize}}}

\newcommand{\argmin}{\mathop{\textrm{arg\,min}}}

\newcommand{\R}{\mathbf{R}} 
\newcommand{\E}{\mathbf{E}}

\setcellgapes{3pt} 

\newcommand{\tu}{\tikz{\draw[black] (0,0) -- (1ex,0) -- (0.5ex,1ex) -- cycle;}}

\newcommand{\td}{\tikz{\draw[black] (0ex,1ex) -- (1ex,1ex) -- (0.5ex,0ex) -- cycle;}} 

\newcommand{\cc}{\tikz{\draw[black] (0.5ex,0.5ex) circle (0.5ex);}} 

\title{An Adaptive Stochastic Gradient Method \\ with Non-negative Gauss-Newton Stepsizes}

\author{Antonio Orvieto\thanks{
ELLIS Institute Tübingen, Max Planck Institute for Intelligent Systems, Tübingen AI Center, Tübingen, Germany. Work partially carried out at Meta in Seattle, USA.
Email: \texttt{antonio@tue.ellis.eu}}\and Lin Xiao\thanks{
Fundamental AI Research (FAIR) at Meta, Seattle, USA. Email: \texttt{linx@meta.com}}
}
\date{}

\begin{document}
\maketitle
\setlength{\marginparwidth}{2cm}

\begin{abstract}
We consider the problem of minimizing the average of a large number of smooth but possibly non-convex functions. In the context of most machine learning applications, each loss function is non-negative and thus can be expressed as the composition of a square and its real-valued square root. This reformulation allows us to apply the Gauss-Newton method, or the Levenberg-Marquardt method when adding a quadratic regularization. The resulting algorithm, while being computationally as efficient as the vanilla stochastic gradient method, is highly adaptive and can automatically warmup and decay the effective stepsize while tracking the non-negative loss landscape. We provide a tight convergence analysis, leveraging new techniques, in the stochastic convex and non-convex settings. In particular, in the convex case, the method does not require access to the gradient Lipshitz constant for convergence, and is guaranteed to never diverge. The convergence rates and empirical evaluations compare favorably to the classical (stochastic) gradient method as well as to several other adaptive methods. 
\end{abstract}

\tableofcontents


\section{Introduction}
We consider the problem of minimizing the average of a large number of loss functions:
\begin{equation}\label{eq:erm}
\minimize_{x\in\R^d} ~f(x):=\frac{1}{N}\sum_{i=1}^N f_i(x),
\end{equation}
where each $f_i$ is assumed to be lower bounded, differentiable and have Lipschitz-continuous gradients. Specifically, we assume that for each~$i$, there exist a constant $L_i$ such that 
\[
\|\nabla f_i(x)-\nabla f_i(y)\| \leq L_i \|x-y\|,  \qquad \forall\,x,y\in\R^d.
\]
Consequently, the average gradient $\nabla f$ has a Lipschitz constant $L\leq (1/N)\sum_{i=1}^N L_i$.

In machine learning applications, each~$f_i$ corresponds to the loss function associated with a data point or the average of a mini-batch \citep[e.g.,][]{bottou2018optimization}. For large~$N$, the cost of frequent averaging over all data points can be prohibitive, therefore the method of choice (in both convex and non-convex settings) is often some variant of Stochastic Gradient Descent~(SGD):
\begin{equation}\label{eq:sgd}
x^{k+1} = x^k - \gamma_k \nabla f_{i_k}(x^k),
\end{equation}
where $i_k\in\{1,2,\dots,N\} =:[N]$ is a randomly picked data point or the index of a mini-batch, and $\gamma_k>0$ is the stepsize~(also known as the \textit{learning rate}) selected at iteration $k$.

The simplest choice is a constant stepsize $\gamma_k=\gamma$ for all $k\geq 0$. If~$\gamma$ is sufficiently small, say $\gamma<2/L$ under the smoothness assumption, then convergence to a neighborhood of a stationary point can be established \citep[e.g.,][]{ghadimi2013stochastic}. However, the Lipschitz constant~$L$ is often unknown and hard to estimate.
Under a (deterministic) decreasing stepsize rule such as $\gamma_k =\gamma_0/\sqrt{k}$, convergence to a local minimizer is eventually recovered for any $\gamma_0>0$ since $\gamma_k<2/L$ after some iterations. Nevertheless, in practice, tuning $\gamma_0$ is still needed to avoid numerical overflows and instabilities, and to obtain optimal performance.
Indeed,~\cite{yang2024two} refined the classical result by~\cite{ghadimi2013stochastic} to show that averaged gradients of the SGD iterates converge to zero for any $\gamma_0>0$, but it takes exponential time to recover from a poorly tuned~$\gamma_0$. 

\subsection{Adaptive stepsizes}
\label{sec:adaptive}

Rather than relying on the knowledge of problem-dependent constants such as~$L$,
adaptive stepsizes exploit additional information from the online iterates $x^k$ and $\nabla f_{i_k}(x^k)$ to obtain faster convergence.
For example, \cite{Kesten58} proposed to adjust $\gamma_k$ based on the sign of the inner product of consecutive stochastic gradients, an idea further developed by \cite{MirzoakhmedovUryasev83} and \cite{DelyonJuditsky93}.
Similar techniques have also been investigated in the machine learning community \citep[e.g.,][]{Jacobs88DBD,Sutton92IDBD,Schraudolph1999, MahmoodSutton12TuningFree}.

\paragraph{AdaGrad and variants.}
Stemming from the seminal work of AdaGrad by \cite{duchi2011adaptive}, modern variants of adaptive stochastic gradient methods employ coordinate-wise stepsizes.
Specifically, instead of a single stepsize~$\gamma_k\in\R_+$, one maintains a vector $\Gamma^k\in\R_+^d$ and the SGD update takes the form 
\begin{equation}\label{eq:coord-sgd}
x^{k+1} = x^k - \Gamma^k\odot \nabla f_{i_k} (x^k),
\end{equation}
where $\odot$ denotes the element-wise product of two vectors. 
Since in this paper we focus on adapting the scalar stepsize~$\gamma_k$ in~\eqref{eq:sgd}, it's more convenient to consider a scalar variant of AdaGrad, called AdaGrad-norm \citep{ward2020adagrad}, which sets the stepsize in~\eqref{eq:sgd} as
\begin{equation}\label{eq:adagrad-norm}
\gamma_k = \frac{\eta}{\sqrt{\delta_0^2+\sum_{\tau=1}^k \|\nabla f_{i_\tau}(x^\tau)\|^2}},
\end{equation}
where $\eta$ and $\delta_0$ are hyper-parameters to be tuned.
(The original AdaGrad sets the coordinate-wise stepsizes $\Gamma^k_j$ in~\eqref{eq:coord-sgd} by replacing the gradient norm $\|\nabla f_{i_\tau}(x^\tau)\|$ in~\eqref{eq:adagrad-norm} with its $j$th element.)
The cumulative sum across iterations in the denominator of~\eqref{eq:adagrad-norm} implies that the stepsize decreases monotonically, a feature shared with simple stepsize rules such as $\gamma_0/\sqrt{k}$ and is key to obtaining the classical $O(1/\sqrt{k})$ convergence rate in the convex case \citep{duchi2011adaptive,reddi2019convergence}. 
However, monotonically decreasing stepsizes cannot efficiently navigate complex landscapes with varying local curvatures and they often lead to very slow convergence in practice.  

More sophisticated adaptive methods \citep[e.g.,][]{sutskever13momentum} replace $\nabla f_{i_k}(x^k)$ in~\eqref{eq:sgd} with the exponential moving average of stochastic gradients (often called momentum). In addition, Adam \citep{kingma2014adam} replaces the cumulative sum in the denominator of~\eqref{eq:adagrad-norm} with their exponential moving average (coordinate-wise). 
In particular, the latter change allows the stepsize to decrease or increase depending on the local features of the loss landscape, which contributes to its better performance for training deep-learning models \citep[e.g.,][]{schmidt2021descending}.

\paragraph{Stochastic Polyak stepsizes.}
A classical non-monotone adaptive stepsize rule is the Polyak stepsize \citep[\S5.3]{polyak87introduction}. For the full gradient method $x^{k+1}=x^k-\gamma_k\nabla f(x^k)$, it sets 
\begin{equation}\label{eq:polyak-stepsize}
\gamma_k = \frac{f(x^k)-f^*}{\|\nabla f(x^k)\|^2},
\end{equation}
where $f^*=\inf_{x\in\R^d} f(x)$. Here, $\nabla f(x^k)$ can be a subgradient if~$f$ is nondifferentiable. Although \citet{polyak87introduction} derived it in the context of convex, nonsmooth optimization, \citet{hazan2019revisiting} showed that it achieves the optimal convergence rates of gradient descent for minimizing smooth and/or strongly convex functions as well, without a priori knowledge of the smoothness and strong convexity parameters. However, a crucial issue of the Polyak stepsize is that the optimal value~$f^*$ needs to be known beforehand, which severely limits its application in practical settings.

Nevertheless, there has been, in the last few years, a rising interest in this method in deep learning, where often $f^*$ is zero---linked to over-parametrization~\citep{zhang2021understanding}. There are several adaptations to the Polyak stepsize to stochastic gradient methods~\citep{rolinek2018l4, berrada2020training, prazeres2021stochastic}, with the most theoretically sound being the variant of~\citet{loizou2021stochastic}, proposing the stochastic Polyak stepsize (SPS) 
\begin{equation}\label{eq:stoch-polyak-stepsize}
\gamma_k = \min\left\{\frac{f_{i_k}(x^k)-f_{i_k}^*}{c\|\nabla f_{i_k}(x^k)\|^2},~\gamma_b\right\},
\end{equation}
where $c,\gamma_b\in\R_{\ge0}$ are hyperparameters and $f_{i_k}^*=\inf_{x\in\R^d} f_{i_k}(x)$ is the minimum valued of $f_{i_k}$. 
While $f_{i_k}$ is in principle unknown, in the setting with small batch sizes and no regularization, it is clear that one can set $f_{i_k}^*=0$. 
\citet{loizou2021stochastic} provided a thorough analysis of this method, showing that convergence to a neighborhood of the solution can be achieved without knowledge of the gradient Lipschitz constant~$L$.

Even in the over-parametrized setting, as soon as regularization is added, we always have $f^*_{i_k}>0$ and its exact value is hard to estimate\footnote{If the problem is overparametrized~($d\gg N$), often $f(x)=0$ can be achieved by some $x^*\in\R^d$. However there exist no $x$ such that, e.g. $f(x) + \frac{\lambda}{2}\|x\|^2=0$ (L2 regularization). \citet{loizou2021stochastic} discusses this issue.}.
In addition, a severe limitation of SPS is that the convergence guarantee in the stochastic setting is for a neighborhood that is independent of the hyperparameters. As such, \textit{no convergence to an arbitrary suboptimality can be guaranteed}. \citet{orvieto2022dynamics} addressed this problem by proposing DecSPS, a variant that gradually decreases~$\gamma_b$ in~\eqref{eq:stoch-polyak-stepsize} and can replace $f^*_{i_k}$ by a lower bound, and proved its convergence to the actual problem solution~(i.e., with arbitrary precision) without knowledge of~$L$. However, their analysis requires strong convexity to effectively bound the iterates.

\paragraph{Other related work.}
Another related line of work is on parameter-free methods for online optimization
\citep{orabona2016coin, orabona2017training, cutkosky2020free} and the D-adaptation methods \citep{defazio2023learning, mishchenko2024prodigy}. Given a fixed number of steps to run, \cite{zamani2023exact} showed that the subgradient method with a linearly decaying stepsize enjoys last-iterate convergence, and \cite{defazio2023when} extended the result to stochastic gradient methods with an online-to-batch conversion.
Very recently, \cite{defazio2024road} proposed a method that combines the Polyak-Ruppert averaging with momentum and demonstrated promising performance gains for both convex optimization and deep learning.



\subsection{Contributions and outline}
\label{sec:contr}

In this paper, we propose a new adaptive stepsize strategy that shows promising performance compared to previous approaches both in theory and in practice. 
The basic idea is to exploit the nonnegativity of the loss functions (ubiquitous in machine learning) by expressing them as the composition of a square and their real-valued square roots. 
This reformulation allows us the apply the classical Gauss-Newton method \citep[e.g.,][\S10.3]{nocedal06book}, or the Levenberge-Marquardt method when adding a quadratic regularization \citep{levenberg44,marquardt63}. 
By minimizing the Nonnegative Gauss-Newton (NGN) estimates of $f_{i_k}$, we derive the NGN stepsize
\begin{equation}\label{eq:ngn}
    \gamma_k = \frac{\sigma}{1+\frac{\sigma}{2f_{i_k}(x^k)}\|\nabla f_{i_k}(x^k)\|^2},
\end{equation}
where $\sigma>0$ is a regularization parameter. 

\begin{table}[t]
\centering
\makegapedcells 
\includegraphics[width=0.9\textwidth]{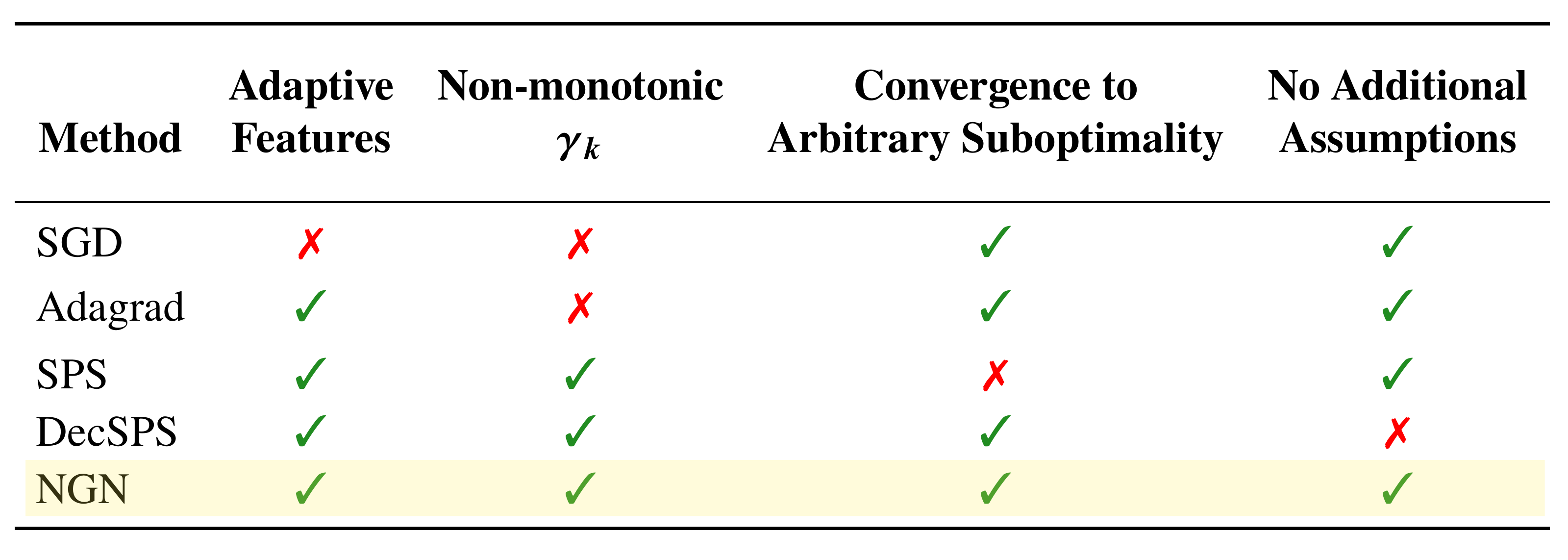}
\caption{Convergence features of SGD and some adaptive methods with convergence guarantees in stochastic convex smooth optimization. Discussion in Sections~\ref{sec:adaptive} and~\ref{sec:contr}.}
\label{tab:optimization_methods}
\end{table}


In Section~\ref{sec:deterministic}, we focus on the deterministic setting (when the full gradient $\nabla f(x^k)$ is available or $N=1$) to derive the NGN stepsize and study its basic properties:
\begin{itemize}\itemsep 0pt
    \item Our derivation reveal that NGN has a clear connection to second-order methods, while sharing the same low computational cost as SGD (Section~\ref{sec:derivation}).
    \item NGN is stable to the selection of its hyperparameter $\sigma$: empirically, non-diverging dynamics are observed for any choice of $\sigma$---a feature that sets the method drastically apart from vanilla SGD that we prove precisely in the convex stochastic setting~(Section~\ref{sec:properties}).
    \item NGN interpolates between an agnostic version of the Polyak stepsize (setting $f^*=0$) and a small constant stepsize (Section~\ref{sec:polyak}), leading to an adaptive \emph{non-monotonic} behavior that can effectively \emph{warm-up} in the beginning and then settle down near the solution (Section~\ref{sec:exp_det}).
\end{itemize}
In Section~\ref{sec:general}, we present an alternative derivation of NGN from a generalized Gauss-Newton perspective, further revealing its curvature adaptation property.
In Section~\ref{sec:stochastic}, we present our main results on convergence analysis in the stochastic setting:
\begin{itemize}\itemsep 0pt
    \item For minimizing smooth convex loss functions~(Section~\ref{sec:cvx-sc-rates}), we show that NGN stepsize can converge to a neighborhood of arbitrary size around the solution\footnote{For SPS~\citet{loizou2021stochastic} the convergence neighborhood is independent of the hyperparameters.} \emph{without knowledge of the gradient Lipschitz constant} and without assumptions such as bounded domain or iterates. The method also provably never diverges.
    Compared to SPS, it behaves much more favorably under no knowledge of $f^*$ or the single $f_{i_k}^*$s. Compared to Adagrad variants, our analysis can be extended to the strongly convex setting with \emph{fully-adaptive linear rate}.
    \item For non-convex smooth optimization (Section~\ref{sec:nonconvex}), NGN stepsize is guaranteed to converge at the same speed as SGD with constant stepsize. In both the convex and non-convex settings, we can gradually decrease the hyperparameter~$\sigma$ and obtain asymptotic convergence in the classical sense (Section~\ref{sec:proofs_annealed_sigma}).
    \item In Section~\ref{sec:experiments}, we conduct extensive numerical experiments on deep learning tasks. We found that NGN stepsize has clear advantage over SGD with constant stepsize, SPS, and AdaGrad-norm in terms of convergence speed and robustness against varying hyperparameters. 
    While our main focus is on adaptive scalar stepsizes for the basic SGD method~\eqref{eq:sgd}, we also conduct experiments to compare with the popular Adam method \citep{kingma2014adam}, which uses both momentum and coordinate-wise stepsizes. Our results show that NGN performs better than Adam in minimizing the training loss and is more robust to hyperparameter tuning, not to mention that NGN takes much less memory to implement.
\end{itemize}
In Section~\ref{sec:conclusion}, we conclude the paper and  point to several interesting directions for future work.

\section{Deterministic setting: derivation and basic properties}
\label{sec:deterministic}

We derive our algorithm when the full gradient $\nabla f(x^k)$ is available (Section~\ref{sec:derivation}), which corresponds to the case of $N=1$ in~\eqref{eq:erm} or when $N$ is relatively small.
Most of our geometric intuition about NGN is presented in the deterministic setting, in Section~\ref{sec:properties}.
In Section~\ref{sec:general}, we provide a more general perspective related to the generalized Gauss-Newton method.

\subsection{Algorithm derivation}
\label{sec:derivation}
 Assume $f:\R^d\to\R$ is differentiable and \textit{non-negative}, i.e., $$f^* := \inf_{x\in\R^d}f(x)\geq 0.$$
 We define $r(x):=\sqrt{f(x)}$ and consequently have $f(x) = r^2(x)$. Therefore, 
 $$\nabla f(x)=2r(x)\nabla r(x) \qquad\text{and}\qquad \nabla r(x) = \frac{1}{2\sqrt{f(x)}}\nabla f(x).$$
The Gauss-Newton update leverages a first-order Taylor expansions of $r(x+p)$ around $x$:
\[
    f(x+p) = r^2(x+p) \simeq \left(r(x)+ \nabla r(x)^\top p\right)^2.
\]
We use this approximation to estimate $f(x+p)$ around $x$ and propose an update $p$ by minimizing
\begin{equation}\label{eq:NGN-approx}
\tilde{f}_{\sigma}^\text{\tiny \ NGN}(x+p) := \left(r(x)+ \nabla r(x)^\top p\right)^2 + \frac{1}{2\sigma}\|p\|^2,
\end{equation}
where $\|\cdot\|$ is the Euclidean norm and $\frac{1}{2\sigma}\|p\|^2$ is a regularization term expressing the confidence $\sigma>0$ in the estimate $\tilde{f}_{\sigma}^\text{\tiny \ NGN}$. Crucially, we note that \textit{this approximation preserves non-negativity}, in contrast to the Taylor expansion of $f$ \citep{chen2011hessian}. The minimizer of $\tilde{f}_{\sigma}^\text{\tiny \ NGN}(x+p)$ with respect to $p$ can be found by setting $\nabla_{\!p} \tilde{f}_{\sigma}^\text{\tiny \ NGN}(x+p)=0$, i.e.,
\begin{align*}
    \nabla_{\!p} \tilde{f}_{\sigma}^\text{\tiny \ NGN}(x+p) 
    &= 2 \left(r(x)+ \nabla r(x)^\top p\right) \nabla r(x) + \frac{1}{\sigma} p\\ &= 2r(x) \nabla r(x) + \left[2\nabla r(x) \nabla r(x)^\top+\frac{1}{\sigma} I\right] p \overset{!}{=}0,
\end{align*}
where $I$ is the $d\times d$ identity matrix. Therefore $p\in\R^d$ needs to satisfy the normal equation 
\[
    \left[\nabla r(x) \nabla r(x)^\top + \frac{1}{2\sigma}I\right] p = - r(x) \nabla r(x).
\]
This equation can be solved analytically using the Sherman-Morrison formula\footnote{Let $A$ be an invertible matrix in $\R^{d\times d}$ and $u,v\in\R^d$, then $(A+uv^\top)^{-1} = A^{-1} -\frac{A^{-1}uv^\top A^{-1}}{1+v^\top A^{-1} u}$.}:
\begin{align*}
     p &\,=\, -\left[\nabla r(x) \nabla r(x)^\top + \frac{1}{2\sigma}I\right]^{-1} r(x) \nabla r(x)
     \;=\; - \frac{2\sigma r(x) \nabla r(x)}{1+2\sigma\|\nabla r(x)\|^2}.
\end{align*}
Now recall that $r(x)=\sqrt{f(x)}$ and $\nabla r(x) = \frac{1}{2\sqrt{f(x)}}\nabla f(x)$, so that $\|\nabla r(x)\|^2 = \frac{1}{4f(x)}\|\nabla f(x)\|^2$ and $r(x)\nabla r(x) = \frac{1}{2}\nabla f(x)$. Therefore, we get
\begin{equation}\label{eq:NGN-det-p}
    p^{\text{\tiny NGN}}_{\sigma} = -\frac{\sigma \nabla f(x)}{1 + \frac{\sigma}{2f(x)} \|\nabla f(x)\|^2}.
\end{equation}
The derivation above suggests the update rule $x\leftarrow x+p^{\text{\tiny NGN}}_{\sigma}$, leading to the deterministic NGN (NGN-det) method:
\begin{equation}\label{eq:NGN-det}
    x^{k+1} = x^k -\frac{\sigma}{1 + \frac{\sigma}{2f(x^k)} \|\nabla f(x^k)\|^2} \nabla f(x^k).
\end{equation}
It is a form of gradient descent $x^{k+1}=x^k-\gamma_k\nabla f(x^k)$ with the \emph{adaptive} NGN stepsize
\begin{equation}\label{eq:NGN-det-gamma}
    \gamma_k^\text{\tiny \ NGN}:=\frac{\sigma}{1 + \frac{\sigma}{2f(x^k)} \|\nabla f(x^k)\|^2}.
\end{equation}
In the rest of this section, we study the basic properties and convergence guarantees of NGN-det.

\subsection{Basic properties of NGN}
\label{sec:properties}
We give below an overview of the properties of the NGN stepsizes, here discussed in the deterministic setting but with direct application in the stochastic case. Some of these properties constitute the workhorse of our convergence analysis, in Section~\ref{sec:stochastic}.
However, we do not provide the convergence rates in the deterministic setting here because they can be obtained from the results in Section~\ref{sec:stochastic} by setting $N=1$.

\subsubsection{Non-negative estimation}
The NGN update $p^{\text{\tiny NGN}}_\sigma(x)$ in~\eqref{eq:NGN-det-p} leverages a local first-order \textit{non-negative} estimate of the loss $f$:
\begin{align}
p^{\text{\tiny NGN}}_{\sigma}(x) &= \argmin_{p\in\R^d} \left[ \tilde{f}_{\sigma}^\text{\tiny \ NGN}(x+p) := \biggl(\sqrt{f(x)}+ \frac{1}{2\sqrt{f(x)}}\nabla f(x)^\top p\biggr)^2 + \frac{1}{2\sigma}\|p\|^2\right].
\label{eq:NGN_estimate}
\end{align} 
This feature is fundamental to NGN and draws a clear distinction from vanilla gradient descent~(GD): indeed, the update $x\leftarrow x+ p^{\text{\tiny GD}}$ of GD with constant stepsize~$\gamma$ is the result of minimization of a first-order estimate of $f$ which is not necessarily non-nagative~(as opposed to the original $f$):
\begin{equation}
    p^{\text{\tiny GD}}_\gamma(x) = \argmin_{p\in\R^d} \left[\tilde{f}_{\gamma}^\text{\tiny \ GD}(x+p) := f(x) + \nabla f(x)^\top p + \frac{1}{2\gamma}\|p\|^2\right] = -\gamma\nabla f(x).
\label{eq:GD_estimate}
\end{equation}

The distinction outlined above can be visualized on toy examples in Figure~\ref{fig:NGN_approx_1d}. It is clear from these examples that the NGN update, minimizing a non-negative first-order estimate of $f$, is more conservative compared to GD with stepsize $\gamma=\sigma$~(the NGN regularization hyperparameter), especially in regions where gradients are large. The stepsize range of NGN will be characterized more precisely in Section~\ref{sec:step_range}, and is linked to the adaptive nature of our method, described later in Section~\ref{sec:polyak}. We note that it is precisely this conservative nature that allows the convergence rates in Section~\ref{sec:stochastic} to hold for an arbitrarily large $\sigma$---in contrast to gradient descent and most adaptive methods, which diverge for large stepsize parameters.


\begin{figure}[t]
    \centering
    \includegraphics[width=0.87\textwidth]{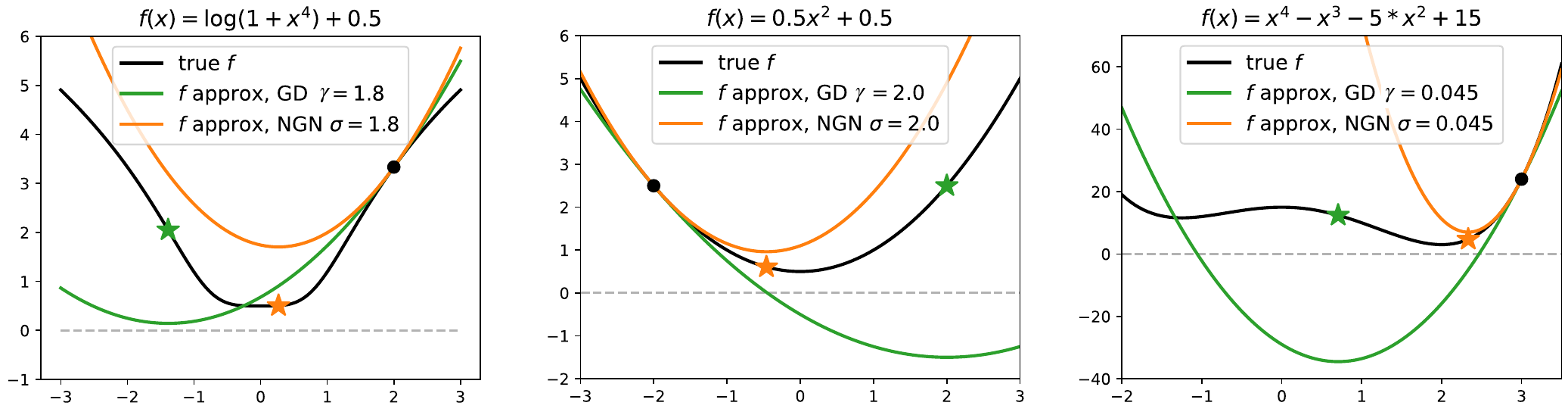}
    \caption{NGN and GD updates and corresponding objective function estimate in~Eqs.~\eqref{eq:NGN_estimate} and~\eqref{eq:GD_estimate} on a few toy examples~(inspired by~\cite{chen2011hessian}). The black dot denotes the initial~$x$, ald the star is the position after one step: $x+p$. Compared to GD with stepsize $\gamma=\sigma$, NGN is more conservative if the landscape is sharp. Note that the function approximation provided by NGN is always non-negative, as clear from  our motivation and the algorithm derivation.}
    \label{fig:NGN_approx_1d}
\end{figure}

\subsubsection{Range of NGN stepsize}
\label{sec:step_range}
Suppose that $f:\R^d\to \R$ is $L$-smooth and has minimum value $f^*$. 
Then for all $x\in\R^d$, we have \citep[e.g.,][\S2.1]{nesterov2018lectures}.
\[
    2L(f(x)-f^*) \ge \|\nabla f(x)\|^2.
\]
Since $f^*\ge0$, we obviously have
\( 2Lf(x) \ge \|\nabla f(x)\|^2 \), which is equivalent to
$$0\le\frac{\|\nabla f(x)\|^2}{2f(x)}\le L.$$ 
These bounds directly imply a range for $\gamma_k$, as characterized in the following lemma.

\begin{lemma}[Stepsize bounds]
\label{lemma:step_bounds}
Suppose $f:\R^d\to\R$ is non-negative, differentiable and $L$-smooth. Then the NGN-det stepsize given in~\eqref{eq:NGN-det-gamma} satisfies
\begin{equation}
    \gamma_k^{\text{\tiny \emph{NGN}}} \in \left[\frac{\sigma}{1+\sigma L},\sigma\right] = \left[\frac{1}{L +\sigma^{-1}},\sigma\right].
\end{equation}
\end{lemma}

This property shows that the maximum achievable stepsize is bounded by~$\sigma$, but the algorithm can adaptively decrease it until $1/(L+\sigma^{-1})$ if the landscape gets more challenging. 




\subsubsection{Connection with Polyak stepsize}
\label{sec:polyak}
The reader can readily spot a connection between the NGN-det stepsize~\eqref{eq:NGN-det-gamma} and the Polyak stepsize~\eqref{eq:polyak-stepsize}.
To be more precise, given $f^*\geq 0$, we define the $f^*$-agnostic Polyak stepsize (APS) 
\begin{equation}\label{eq:agnostic-PS}
    \gamma_k^\text{\tiny APS} = \frac{f(x^k)}{\|\nabla f(x^k)\|^2}.
\end{equation}
We can express the NGN stepsize as the harmonic mean of APS and the constant stepsize $\sigma/2$:
\[
\gamma_k^{\text{\tiny NGN}} 
~=~ \frac{2}{\frac{2}{\sigma} + \frac{\|\nabla f(x^k)\|^2}{f(x^k)}}
~=~ \frac{2}{\frac{1}{\sigma/2}+\frac{1}{\gamma_k^{\text{\tiny APS}}}}.
\]
Intuitively, as $\sigma$ increases (less regularization in~\eqref{eq:NGN-approx}) , NGN relies more on the Gauss-Newton approximation.  
In the limiting case of $\sigma\to\infty$, we have
$\gamma_k^{\text{\tiny NGN}}\to 2\gamma_k^{\text{\tiny APS}}$, which is two times the APS.
On the other hand, as $\sigma\to 0$, the regularization in~\eqref{eq:NGN-approx} dominates the approximation and we get back to gradient descent with a small constant stepsize $\gamma_k=\sigma$. As such, 
%
\begin{center}
     \textit{NGN interpolates between the $f^*$-agnostic Polyak stepsize and a small constant stepsize.}
 \end{center}

\begin{figure}[t]
    \centering
\includegraphics[width=0.96\textwidth]
{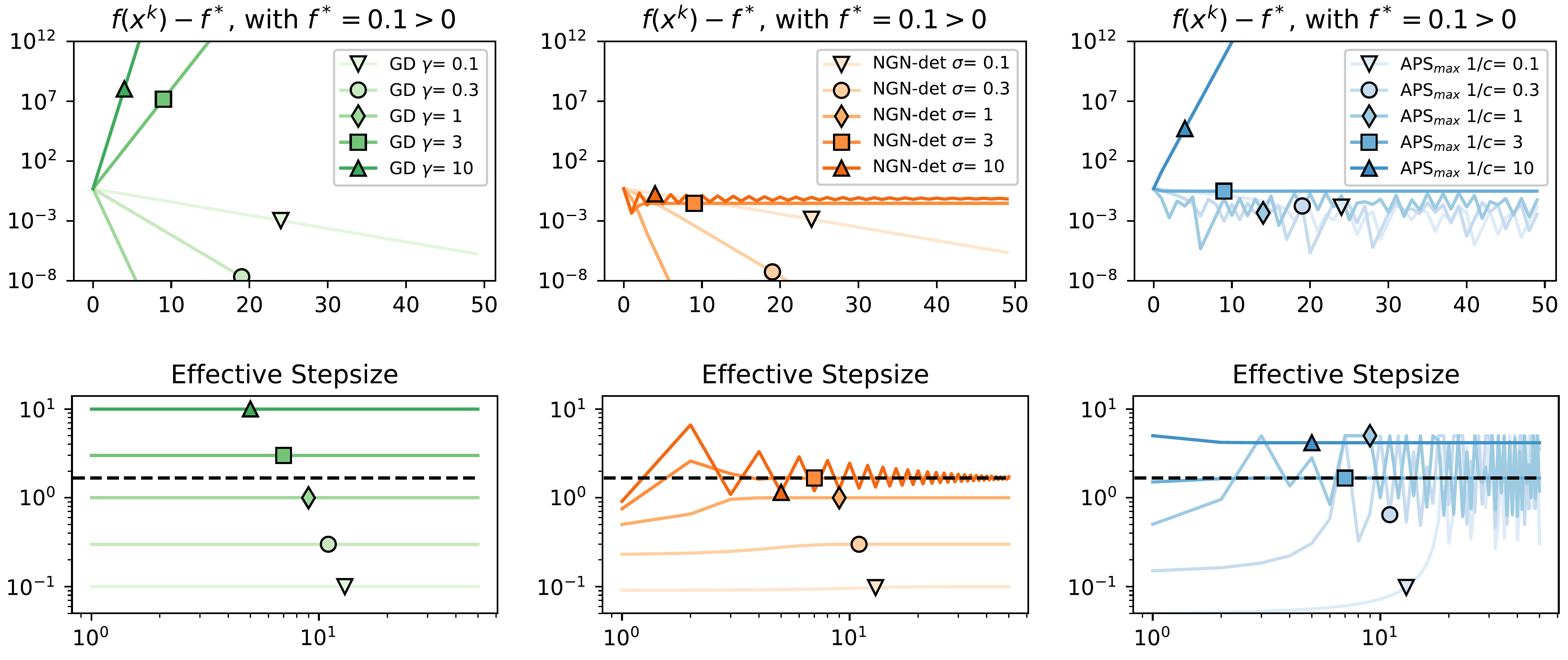}
    \caption{Optimization dynamics of constant-stepsize GD, NGN, and APS$_{\max}$ on the toy example $f(x) = \frac{\lambda}{2} (x-x^*)^2+f^*$, for different hyperparameter values. NGN is stable for any finite $\sigma>0$. Dashed line in the bottom row is the value $2/\lambda$.}
    \label{fig:1d_ngn_sps}
\end{figure}

We like to spend a few more words on the insightful connection between PS and NGN. As already mentioned, one obvious drawback of PS, persisting in its stochastic variant SPS~\citep{loizou2021stochastic}, is the required knowledge of $f^* = \inf_x f(x)$. However, \citet{orvieto2022dynamics} showed that convergence to a ball around the solution, of size proportional to $f^*$, can be retrieved both in the deterministic and stochastic setting by replacing~$f^*$ with any lower bound $\ell^*\le f^*$---in particular $\ell^*=0$, which becomes NGN with $\sigma\to\infty$. It is therefore interesting to compare NGN with the following PS variant inspired by the SPS$_{\max}$ rule~\citep{loizou2021stochastic,orvieto2022dynamics}:
\begin{equation}
    \gamma_k^{\text{\tiny APS}_{\max}}
    = \min\left\{\frac{f(x^k)}{c\|\nabla f(x^k)\|^2},~\gamma_b\right\},
\end{equation}
where $c>0$ is a hyperparameter and $\gamma_b>0$ bounds the maximum stepsize achievable. 
For this purpose, we consider the simplest convex optimization setting with a one-dimensional quadratic $f(x) = \frac{\lambda}{2} (x-x^*)^2+f^*$, where $f^*=0.1>0$ and $\lambda=1.2$. 
Figure~\ref{fig:1d_ngn_sps} shows the loss dynamics for GD, NGN and APS$_{\max}$ with $\gamma_b=3$ (results are not sensitive to $\gamma_b$ when it is big enough).
Ablating on different values of $\sigma$ for NGN, $c$ for APS$_{\max}$ and $\gamma$ for constant-stepsize GD, one notices a striking feature of NGN: its iterates \textit{never grow unbounded} for any value of $\sigma$. Instead, for $\gamma$ big enough or $c$ small enough, GD and APS$_{\max}$ grow unbounded. 
This property of NGN is also observed in our experiments on convex classification~(Fig.~\ref{fig:cvx_det_all}) and neural networks~(Fig.~\ref{fig:nn_res}), and we will give it a formal proof in Section~\ref{sec:stochastic}.
If~$c$ has a high enough~($L$-independent) value, then \citet{orvieto2022dynamics} showed that SPS$_{\text{max}}$ also does not grow unbounded. However, it cannot obtain an arbitrary suboptimality 
(by tuning the hyperparameters of $\gamma_k^\text{\tiny \ APS$_{\max}$}$)---which instead can be achieved by NGN in this setting, as shown by our theory in Section~\ref{sec:stochastic}. This issue of SPS$_{\text{max}}$ is related to its bias problem, thoroughly explored in~\citet{orvieto2022dynamics}.

From Figure~\ref{fig:1d_ngn_sps}, we also observe another important property of NGN with large~$\sigma$: its stepsize converges to a value bounded by $2/\lambda$---the curvature-informed stability threshold for the corresponding dynamics. We remark that both~$\lambda$ and~$f^*$ do not appear in the NGN update rule. This shows a remarkable curvature adaptation property of NGN, which is linked to the derivation in Section~\ref{sec:general}. In contrast, APS$_{\max}$, while in perfect agreement with the findings in~\cite{orvieto2022dynamics} converges to a ball around the solution for big enough hyperparameter $c$, its performance is negatively affected by having no access to $f^*$~(cf. Figure~\ref{fig:cvx_det_all}). 


\subsection{Experiments on convex classification problems}
\label{sec:exp_det}

\begin{table}[p]
    \centering
    \begin{tabular}{|c|c|c|c|c|c|}
        \hline
        \textbf{Dataset} & \textbf{Features Dimension} & \textbf{\# Datapoints} & \textbf{\# Classes} & \textbf{L2 Regularization} \\
        \hline
        Cancer & 30 & 569 & 2 & $1 \times 10^{-4}$ \\
        \hline
        Wine & 13 & 178 & 3 & $1 \times 10^{-4}$\\
        \hline
        Digits & 64 & 1797 & 10 & $1 \times 10^{-4}$\\
        \hline
    \end{tabular}
    \caption{LIBSVM~\citep{chang2011libsvm} dataset information, results in Figure~\ref{fig:cvx_det_all}.}
    \label{tab:dataset}
\end{table}

\begin{figure}[p]
    \centering
    \includegraphics[width=\textwidth]{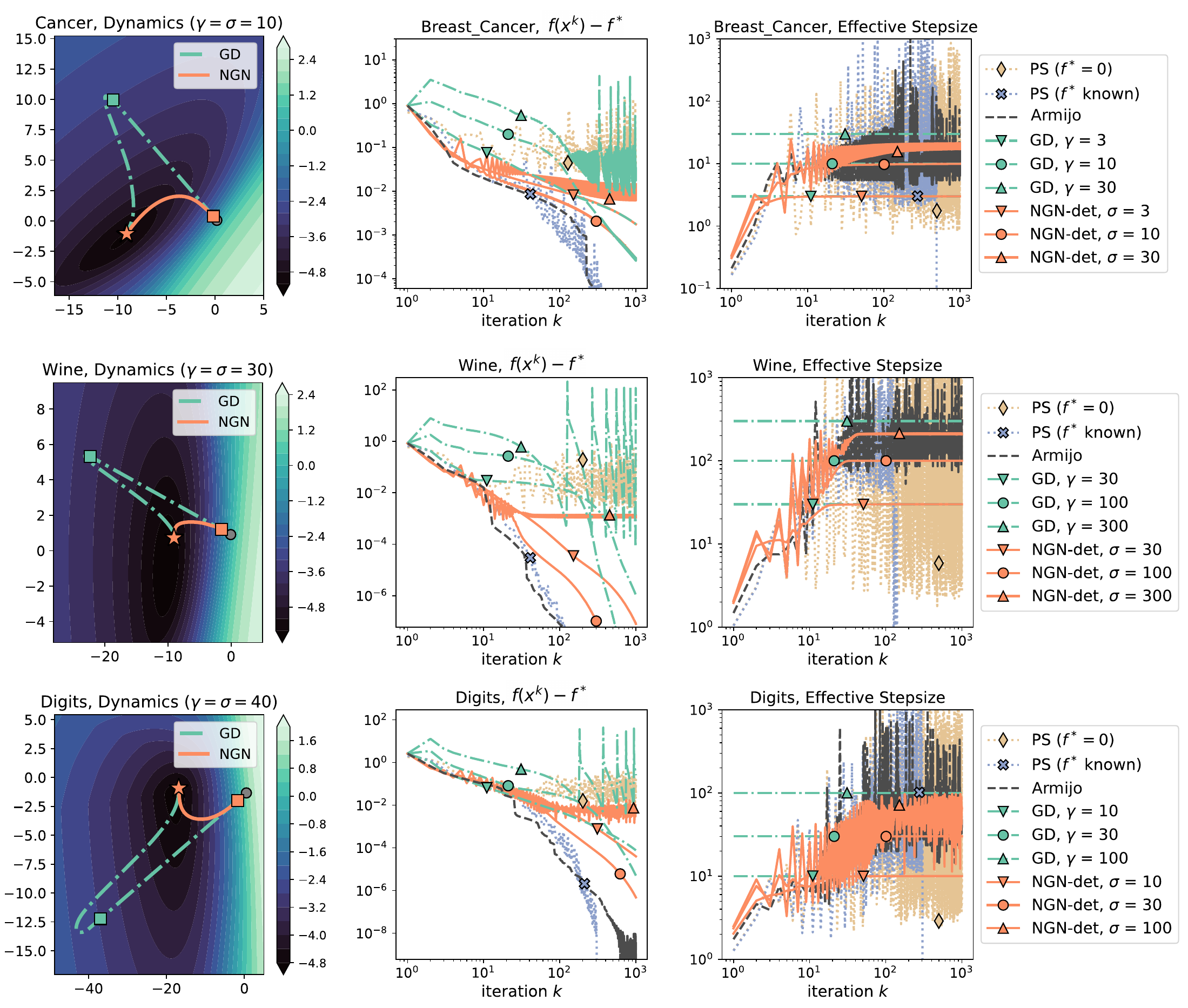}
    \small
    \caption{Deterministic NGN compared with constant-stepsize Gradient Descent, Polyak Stepsizes, and Armijo Line search on three classification datasets listed in Table~\ref{tab:dataset}. The center column shows the optimality gap $f(x^k)-f^*$, the right column plots the evolution of stepsizes, and the right column shows the projection of the iterate trajectories onto the top-two PCA components.
    On the trajectories plot, the circle denotes the starting point, the star denotes the solution found at the last iteration, and the square represents the point after one iteration.}
    \label{fig:cvx_det_all}
\end{figure}

In Figure~\ref{fig:cvx_det_all}, we evaluate the \textit{full-batch} training performance of NGN-det on convex classification~(using the cross-entropy loss). Specifically, we consider three classification datasets from LIBSVM~\citep{chang2011libsvm}, whose details are given in Table~\ref{tab:dataset}.

We compare NGN with constant-stepsize gradient descent (GD), standard Polyak stepsize~(PS) with or without knowledge of $f^*>0$, and GD with Armijo line search~(LS).  
For GD and NGN-det, we sweep over hyperparameters~($\gamma$ and $\sigma$) on a logarithmic grid and denote the best-performing option with marker \cc. Performace for 3-times-lower and 3-times-higher hyperparameters is marked by symbols \td \  and \tu \ , respectively. 
For all PS variants, we select $c=1$~(see Section~\ref{sec:polyak}), as suggested by the results in~\citet{loizou2021stochastic}.

In Figure~\ref{fig:cvx_det_all}, we plot the loss function optimality gap, the stepsize, as well as the projection of the iterate dynamics onto the top-two principle components. They clearly showcase NGN's adaptivity to large gradients at the beginning and its stability to high values of~$\sigma$ (cf. Section~\ref{sec:polyak}). Moreover, 
the dynamics of the NGN stepsize exhibit a warmup phase followed by convergence to a flat value towards the end ($\gamma_k\to\sigma$ if~$\sigma$ is not too large);
see bounds in Section~\ref{sec:step_range}. A similar behavior is observed for the line-search stepsize~(LS) and for PS using the correct value of~$f^*$ (here $f^*>0$ because we have more datapoints than dimensions; see Table~\ref{tab:dataset}). Of course, in practice $f^*$ is not accessible beforehand. Plugging in the lower bound $f^*=0$ still guarantees PS approaches the minimal loss~(see discussion in Section~\ref{sec:polyak}), but the iterates do not converge to the true solution---in contrast to NGN~(See Figure~2 in~\citet{orvieto2022anticorrelated}). 

\section{A generalized Gauss-Newton perspective}
\label{sec:general}
Before moving on to the stochastic case and presenting the  convergence analysis,
we provide here an alternative derivation of the NGN method---with increased generality compared to the one given in Section~\ref{sec:derivation}. This generalized Gauss-Newton (GNN) perspective allows us to make a closer connection to second-order methods and better understand the curvature adaptation property of NGN demonstrated in Figures~\ref{fig:1d_ngn_sps} and~\ref{fig:cvx_det_all}.

Suppose $f:\R^d\to\R$ can be written as the composition of two functions: $f(x) = h(c(x))$, where $h:\R\to\R$ is differentiable and nonnegative and $c:\R^d\to\R$ is differentiable. 
In the setting of Section~\ref{sec:derivation}, we have $h(c) = c^2$ and $c(x) = r(x) = \sqrt{f(x)}$. The gradient of~$f$ is
\begin{align*}
\nabla f(x) &= h'(c(x)) \nabla c(x), 
\end{align*}
and its Hessian can be written as
\begin{align}
\nabla^2 f(x) &= h'(c(x)) \nabla^2 c(x) + h''(c(x)) \nabla c(x)\nabla c(x)^\top \nonumber \\
&= \underbrace{h'(c(x)) \nabla^2 c(x)}_{\text{Difficult to compute}} + \underbrace{\frac{h''(c(x))}{h'(c(x))^2}\nabla f(x)\nabla f(x)^\top}_{\text{Easy to compute}},
\label{eq:ggn-hessian}
\end{align}
where in the last equality we applied the substitution $\nabla c(x)=\nabla f(x)/h'(c(x))$.
Notice that $h'(c(x))=0$ if and only if $h$ is minimized at $c(x)$ and thus $f$ is minimized at~$x$.
Therefore we can assume without loss of generality that $h'(c(x))\neq 0$.

We have seen in~\eqref{eq:GD_estimate} that the GD update $p^{\text{GD}}_\sigma$ minimizes the first-order approximation $\tilde{f}^{\,\text{GD}}_\sigma$.
Similarly, the canonical second-order method, Newton's method, derive the update as
\[
p^{\,\textrm{Newton}} = \argmin_p \left\{ \tilde{f}^{\,\textrm{Newton}}(x+p) := f(x) + \nabla f(x)^\top p + \frac{1}{2}p^\top \nabla^2 f(x) p \right\},
\]
where $\nabla^2 f(x)$ is given in~\eqref{eq:ggn-hessian}. In order to avoid the ``difficult-to-compute'' part of the Hessian, the generalized Gauss-Newton (GGN) method replaces it with a simple approximation $\frac{1}{\sigma}I$:
\begin{equation}\label{eq:ggn-approx}
\tilde{f}^{\,\textrm{GGN}}_\sigma(x+p) := 
f(x) + \nabla f(x)^\top p + \frac{1}{2}p^\top \left[ \frac{1}{\sigma} I + q(x) \nabla f(x)\nabla f(x)^\top \right] p.
\end{equation}
where 
\[
q(x) := \frac{h''(c(x))}{h'(c(x))^2}.
\]
The GGN update $p^{\text{GGN}}$ is the minimizer of 
$\tilde{f}^{\,\textrm{GGN}}_\sigma(x+p)$, which we obtain by forcing its gradient with respect to~$p$ to be zero:
%
\begin{equation}
    \nabla_p \tilde f^{\,\text{GGN}}_\sigma(x+p) 
    = \nabla f(x) + \left[ \frac{1}{\sigma}I + q(x) \nabla f(x)\nabla f(x)^\top \right] p \overset{!}{=} 0,
    \label{eq:GGN_drop}
\end{equation}
and it results in
\begin{equation}\label{eq:ggn-update}
p^{\text{GGN}} = - \frac{\sigma}{1+\sigma q(x)\|\nabla f(x^k)\|^2} \nabla f(x^k).
\end{equation}
%
Next we give several examples of the GGN update by choosing a specific function $h$.
\begin{itemize}
\item \textbf{Quadratic.} In this case, we have $h(c)=c^2$, $h'(c)=2c$ and $h''(c)=2$, thus 
$$q(x)=\frac{1}{c(x)^2} = \frac{1}{2f(x)},$$
and the GGN update~\eqref{eq:ggn-update} becomes the NGN update~\eqref{eq:NGN-det-p}.
In this case, the GGN approximation~\eqref{eq:ggn-approx} reduces exactly to the NGN approximation~\eqref{eq:NGN_estimate}.
\item \textbf{Monomial.}
For a generic $\alpha$-scaled monomial $h(c) = \frac{\alpha}{p}c^p$ we have $h'(c) = \alpha c^{p-1}$ and $h''(c) = \alpha (p-1) c^{p-2}$. Therefore,
\[
    \frac{h''(c)}{h'(c)^2} = \frac{\alpha(p-1) c^{p-2}}{\alpha^2 c^{2p-2}} = \frac{p-1}{\alpha c^p} = \frac{1}{\frac{p}{p-1} h(c)}\quad \implies \quad q(x) = \frac{1}{\frac{p}{p-1} f(x)}.
\]
For $p=2$, we get back to the NGN update.
\item \textbf{Negative logarithm.}
If $h(y) = -\log(y)$, with $y\in(0,1)$ being \textit{the model likelihood}, then $h'(y) = -1/y$ and $h''(y) = 1/y^2$, therefore $q(x) = 1$ and we have
$$x^{k+1} = x^k -\frac{\sigma}{1+\sigma\|\nabla f(x)\|^2}\nabla f(x).$$ 
This can be considered as the scalar case of the diagonal Levenberg-Marquardt method when using the log-likelihood loss \citep[\S9.1]{lecun98efficient}.
\end{itemize}
In our empirical study, the choice $h(c)=c^2$ works best among the monomials with different exponent~$p$ and also performs slightly better than the negative logarithm. In the rest of this paper, we focus on the quadratic~$h$ for convergence analysis and presenting the experiment results.

\section{Stochastic setting: convergence rates and experiments}
\label{sec:stochastic}

In this section, we study the application of NGN to stochastic optimization. Specifically, we consider the loss function in~\eqref{eq:erm}
and assume that each $f_i:\R^d\to\R$ is non-negative and $L_i$-smooth. 
In this section, we let $L=\max_{i\in[N]} L_i$ instead of the average of~$L_i$ as in the deterministic setting; alternatively, we can simply assume that each~$f_i$ has the same smoothness constant~$L$. 
We consider the generic SGD method~\eqref{eq:sgd} where the stepsize $\gamma_k$ is obtained by applying the NGN method to the current, randomly picked function $f_{i_k}$, i.e., 
\begin{equation}\label{eq:stoch-NGN}
    \gamma_k = \frac{\sigma}{1+\frac{\sigma}{2f_{i_k}(x^k)} \ \|\nabla f_{i_k}(x^k)\|^2}.
\end{equation}
All of our results are presented in the stochastic setting, and the convergence rates for the deterministic setting can be recovered by simply taking $N=1$.

\subsection{Preliminary lemmas}
We present here a few fundamental lemmas, heavily used both in the constant $\sigma$ setting~(Section~\ref{sec:fixed_sigma}) and in the decreasing $\sigma$ case~(Section~\ref{sec:proofs_annealed_sigma}).

\begin{lemma}[Fundamental Equality]
\label{lemma:NGN_lemma}
The NGN stepsize $\gamma_k$ in~\eqref{eq:stoch-NGN} satisfies
\begin{equation*}
    \gamma_k \|\nabla f_{i_k}(x)\|^2= 2\left(\frac{\sigma - \gamma_k}{\sigma}\right) f_{i_k}(x).
\end{equation*}
\end{lemma}
\begin{proof}
The definition of our stepsize implies
\[
    \left(1+\frac{\sigma}{2f_{i_k}(x)}\|\nabla f_{i_k}(x)\|^2\right)\gamma_k = \sigma ,
\]
which one can rewrite as
\[
    \frac{\sigma}{2f_{i_k}(x)}\|\nabla f_{i_k}(x)\|^2\gamma_k = \sigma -\gamma_k .
\]
This proves the result.
\end{proof}


\begin{lemma}[Fundamental Inequality]
\label{lemma:fundamental_NGN}
Let each $f_i:\R^d\to\R$ be non-negative, $L$-smooth and convex. Considering the NGN stepsize $\gamma_k$ as in~\eqref{eq:stoch-NGN}, we have
\begin{equation}
    \gamma_k^2 \| \nabla f_{i_k}(x^k)\|^2 ~\le~ \left(\frac{4\sigma L}{1+2\sigma L}\right)\gamma_k [f_{i_k}(x^k) - f_{i_k}^*] + \frac{2\sigma^2 L}{1+\sigma L} \cdot\max\left\{0,\frac{2\sigma L-1}{2\sigma L+1}\right\} \cdot f_{i_k}^*.
\end{equation}
For $\sigma\le1/(2L)$ or $f_{i_k}^*=0$ we have no error resulting from the last term above. 
\end{lemma}
\begin{proof}
Lemma~\ref{lemma:step_bounds}\&\ref{lemma:NGN_lemma} imply the following relations:
\begin{align*}
    \gamma_k^2 \| \nabla f_{i_k}(x^k)\|^2 & ~\le~ 2L\gamma_k^2 (f_{i_k}(x) - f_{i_k}^*),\\
    \gamma_k^2 \| \nabla f_{i_k}(x^k)\|^2 & ~=~ 2\gamma_k\left(\frac{\sigma - \gamma_k}{\sigma}\right) f_{i_k}(x^k).
\end{align*}
They further imply that for any $\delta\in(-\infty,1]$, we have
\begin{align*}
    (1-\delta)\gamma_k^2 \| \nabla f_{i_k}(x^k)\|^2 & \,\le\, 2 (1-\delta)L\gamma_k^2 (f_{i_k}(x) - f_{i_k}^*),\\
    \delta \gamma_k^2 \| \nabla f_{i_k}(x^k)\|^2 & \,=\, 2\delta\gamma_k\left(\frac{\sigma - \gamma_k}{\sigma}\right) f_{i_k}(x^k).
\end{align*}
Summing the two inequalities above, we get
\begin{equation*}
    \gamma_k^2 \| \nabla f_{i_k}(x^k)\|^2\le 2\delta\gamma_k\left(\frac{\sigma - \gamma_k}{\sigma}\right) f_{i_k}(x^k) +2 (1-\delta)L \gamma_k^2(f_{i_k}(x^k) - f_{i_k}^*),
\end{equation*}
and after collecting a few terms,
\begin{equation}
    \gamma_k^2 \| \nabla f_{i_k}(x^k)\|^2\le 2\gamma_k\underbrace{\left[\delta\left(\frac{\sigma - \gamma_k}{\sigma}\right) + (1-\delta)L\gamma_k\right]}_{A(\delta)} [f_{i_k}(x^k) - f_{i_k}^*] + \underbrace{2\gamma_k\delta\left(\frac{\sigma - \gamma_k}{\sigma}\right)}_{B(\delta)} f_{i_k}^*.
    \label{eq:fundamental_lemma_ineq_A_b}
\end{equation}
We would like to choose $\delta$ to somehow minimizes $A(\delta)$ and $B(\delta)$ simultaneously. It turns out that a convenient choice is (see the remark after the proof)
\begin{equation}\label{eq:delta-choice}
    \delta = \frac{2\sigma L-1}{2\sigma L+1} .
\end{equation}
Note that since $2\sigma L>0$, $\delta$ is a real number in the range $[-1,1]$~(which is a valid subset of $[-\infty,1]$). 
Now consider the first term~$A(\delta)$ in~\eqref{eq:fundamental_lemma_ineq_A_b}. We have
\begin{align*}
    A(\delta)&= \frac{2\sigma L-1}{2\sigma L+1}\left(1-\frac{\gamma_k}{\sigma}\right) + \frac{2 L \gamma_k}{2\sigma L+1}\\
    &= \frac{2\sigma L-1}{2\sigma L+1} - \frac{2\sigma L-1}{2\sigma L+1}\frac{\gamma_k}{\sigma} + \frac{2 L \gamma_k}{2\sigma L+1}\\
    &= \frac{2\sigma L-1}{2\sigma L+1} +\frac{\gamma_k}{(2\sigma L+1)\sigma}\\
    &\le \frac{2\sigma L}{2\sigma L+1},
\end{align*}
where in the last line we used the property $\gamma_k\le\sigma$. The bound~\eqref{eq:fundamental_lemma_ineq_A_b} becomes:
\begin{equation*}
    \gamma_k^2 \| \nabla f_{i_k}(x^k)\|^2\le \left(\frac{4\sigma L}{1+2\sigma L}\right)\gamma_k [f_{i_k}(x^k) - f_{i_k}^*] + 2\gamma_k\left(\frac{2\sigma L-1}{2\sigma L+1}\right)\left(\frac{\sigma - \gamma_k}{\sigma}\right) f_{i_k}^*.
\end{equation*}
To bound $B(\delta)$, i.e., the second term on the right side of the above inequality, we have two cases:
\begin{itemize}
    \item If $\sigma\le1/(2L)$ then $\delta = \frac{2\sigma L-1}{2\sigma L+1}$ is negatkve, and we have
    \begin{equation}
        2\gamma_k\left(\frac{2\sigma L-1}{2\sigma L+1}\right)\left(\frac{\sigma - \gamma_k}{\sigma}\right) f_{i_k}^*\le 0,
    \end{equation}
    since the worst case is $\gamma_k=\sigma$.
    \item Otherwise, $\delta>0$ and we can proceed as follows, using the fact that $\gamma_k \in \left[\frac{\sigma}{1+\sigma L},\sigma\right]$
    \begin{align*}
        2\gamma_k\left(\frac{2\sigma L-1}{2\sigma L+1}\right)\left(\frac{\sigma - \gamma_k}{\sigma}\right) f_{i_k}^*&\le 2\sigma\left(\frac{2\sigma L-1}{2\sigma L+1}\right)\left(\frac{\sigma - \frac{\sigma}{1+\sigma L}}{\sigma}\right) f_{i_k}^*\\ &= 2\sigma \left(\frac{2\sigma L-1}{2\sigma L+1}\right)\left(\frac{\sigma L}{1+\sigma L}\right) f_{i_k}^* \\
        &= \frac{2\sigma^2 L}{1+\sigma L} \left(\frac{2\sigma L-1}{2\sigma L+1}\right)f_{i_k}^* ,
    \end{align*}
\end{itemize}
All in all, considering both cases, we get the following upper bound for the error term $B(\delta)$:
\begin{equation}
    \frac{2\sigma^2 L}{1+\sigma L} \cdot\max\left\{0,\frac{2\sigma L-1}{2\sigma L+1}\right\} \cdot f_{i_k}^*.
\end{equation}
This concludes the proof.
\end{proof}

\begin{remark}[On the choice of $\delta$ in~\eqref{eq:delta-choice}]
In the context of the proof of Lemma~\ref{lemma:fundamental_NGN}, let us assume we want to find $\delta$ such that $A(\delta)\le\alpha$. That implies
\begin{align*}
&\delta\left(\frac{\sigma - \gamma_k}{\sigma}\right) + (1-\delta)L\gamma_k\le\alpha\\
\iff & \delta\sigma +\left[(1-\delta)L\sigma-\delta\right]\gamma_k\le\alpha\sigma\\
\iff&\left[(1-\delta)L\sigma-\delta\right]\gamma_k\le(\alpha-\delta)\sigma\\
\iff&\left[(1-\delta)L\sigma-\delta\right]\le(\alpha-\delta)\frac{\sigma}{\gamma_k}.
\end{align*}
For $\alpha\ge\delta$, the right-hand side is positive. Further, note that $\frac{\sigma}{\gamma_k}\in[1,1+\sigma L]$. So if we like the inequality to hold for every value of $\gamma$ we need~(worst case analysis)
\begin{equation*}
    \left[(1-\delta)L\sigma-\delta\right]\le(\alpha-\delta) \iff (1-\delta)L\sigma\le\alpha \iff \delta\ge1-\frac{\alpha}{L\sigma}.
\end{equation*}
Since $\alpha\ge\delta$ we also need
\begin{equation*}
    \alpha\ge1-\frac{\alpha}{L\sigma}\iff L\sigma\alpha\ge  L\sigma -\alpha\iff\alpha\ge \frac{\sigma L}{1+\sigma L} = 1-\frac{1}{1+\sigma L}.
\end{equation*}
The bound becomes
\begin{equation*}
    \gamma_k^2 \| \nabla f_{i_k}(x^k)\|^2\le 2\alpha\gamma_k [f_{i_k}(x^k) - f_{i_k}^*] + \underbrace{2\gamma_k\delta\left(\frac{\sigma - \gamma_k}{\sigma}\right)}_{B(\delta)} f_{i_k}^*.
\end{equation*}
For the sake of minimizing the first term in the bound, it would make sense to use $\alpha = \frac{\sigma L}{1+\sigma L}$. However, under this condition we get that $\delta\ge \frac{\sigma L}{1+\sigma L}$ as well. This is not ideal since we want $B(\delta)$, the error factor, to vanish for small $\gamma$. To do this, we need to slightly increase the value of $\alpha$ to allow $\delta$ to become negative for small values of $y=L\sigma$. Note that for $\delta$ to be (potentially) negative we need
\begin{equation*}
    1-\frac{\alpha}{\sigma L}\le0\iff \alpha\ge L\sigma.
\end{equation*}
Hence, if we want to keep $\alpha<1$~(needed for proofs), we can only have this condition for $\sigma\le 1/L$ -- i.e. the case where we know convergence holds. To this end, let us consider
\begin{equation*}
    1>\alpha= \frac{2\sigma L}{1+2\sigma L}= \frac{\sigma L}{1/2+\sigma L}\ge \frac{\sigma L}{1+\sigma L}.
\end{equation*}
Using this $\alpha$, we get
\begin{equation*}
    \delta\ge1-\frac{\alpha}{L\sigma} = 1-\frac{2}{1+2\sigma L} = \frac{2\sigma L-1}{2\sigma L+1}.
\end{equation*}
Note that if $\sigma\le\frac{1}{2L}$ then the minimum allowed $\delta$ is negative. \textit{We pick for $\delta$ its minimum allowed value}. 
\end{remark}

\subsection{Convergence rates for fixed regularization}
\label{sec:fixed_sigma}
We provide here convergence guarantees for stochastic NGN method in the convex, strongly convex, and non-convex setting. The results justify the performance observed in practice so far, complemented by deep learning experiments in Section~\ref{sec:experiments}.

Our results are better presented with the concept of interpolation. For a finite sum loss $f =\sum_{i=1}^N f_i$, interpolation implies that all $f_i$s can be minimized at the solution.
\begin{definition}[Interpolation]
   The optimization problem $\min_{x\in\R^d} f(x)=\sum_{i=1}^N f_i(x)$ satisfies interpolation 
if there exist a $x^*$ such that $x^*\in\argmin_{x\in\R^d} f_i(x)$ for all $i\in[N]$.
\end{definition}
Interpolation is an assumption that is often used in the deep learning literature, as it is associated with overparametrization of some modern neural networks, leading to fast convergence of gradient methods~\citep{ma2018power}. We do not assume interpolation\footnote{For instance, large language models do not satisfy this property~\citep{hoffmann2022training}.}, but follow \citet{loizou2021stochastic} to provide convergence rates in relation to the following two error quantities:
\begin{equation}
{\color{purple}\Delta_{\text{int}}:=\mathbf{E}[f_{i}(x^*)-f_{i}^*]},\qquad {\color{violet}\Delta_{\text{pos}}:=\mathbf{E}[f_{i}^*]},
\label{eq:errors}
\end{equation}
where $x^*\in\argmin f(x)$ and $f_i^*:=\inf_x f_i(x)$.
Here the expectation $\E[\cdot]$ is taken with respect to the uniform distribution of the index~$i$, equivalent to the average $(1/N)\sum_{i=1}^N[\cdot]$.
We call $\Delta_{\text{int}}$ the interpolation gap, which is apparently zero under interpolation. The position gap $\Delta_{\text{pos}}$ measures how close is, on average, the minimizer for the current batch to the value zero.
While for overparametrized deep learning models one has $\Delta_{\text{pos}}=\Delta_{\text{int}}=0$~\citep{vaswani2020adaptive}, both the cases $\Delta_{\text{pos}}=0,\Delta_{\text{int}}>0$ and $\Delta_{\text{pos}}>0,\Delta_{\text{int}}=0$ are feasible in theory. 

\begin{remark}
    Since in our setting each $f_i$ is lower-bounded, under the very realistic assumptions $f^*<\infty$ and $f_i^*< \infty$ for all $i\in[N]$, both $\Delta_{\text{pos}}$ and $\Delta_{\text{int}}$ are finite. This assumption, in particular, does not imply finite gradient noise variance.
\end{remark}

\subsubsection{Convex and strongly convex settings}
\label{sec:cvx-sc-rates}

We first recall the definition in the differentiable setting.

\begin{definition}[Strong Convexity / Convexity]
A differentiable function $f : \R^d \rightarrow \R$, is $\mu$-strongly convex, if there exists a constant
$\mu > 0$ such that $\forall x, y \in \R^d$:
\begin{equation}
\label{eq:stronglyconvex}
 f(x) \geq f(y)+ \langle\nabla f(y) , x-y\rangle + \frac{\mu}{2} \|x-y\|^2
\end{equation}
for all $x \in \R^d$. If the inequality holds with $\mu=0$ the function $f$ is convex.
\end{definition}

We now present our results in the convex and strongly convex settings. The nonconvex case, for which we have weaker results, is presented in Section~\ref{sec:nonconvex}.

\begin{tcolorbox}
\begin{theorem}[NGN, convex]
\label{thm:stoc_cvx}
Let $f =\frac{1}{N}\sum_{i=1}^{N} f_i$, where ealh $f_i:\R^d\to\R$ is non-negative, 
$L$-smooth and convex. 
Consider the SGD method~\eqref{eq:sgd} with the NGN stepsize~\eqref{eq:stoch-NGN}.
For any value of $\sigma>0$, we have
\begin{align}
    \mathbf{E}\left[f(\bar x^K)-f(x^*)\right] &\le {\color{olive}\frac{\mathbf{E}\|x^0-x^*\|^2}{\eta_{\sigma} K}} + {\color{purple}3\sigma L\cdot (1+\sigma L) \Delta_{\text{int}}} + {\color{violet}\sigma L \cdot\max\left\{0,2\sigma L-1\right\} \Delta_{\text{pos}}},
\end{align}
where $\bar x^K =\frac{1}{K}\sum_{k=0}^{K-1}x^k$ and $\eta_{\sigma} =\frac{2\sigma}{(1+2\sigma L)^2}$. Decreasing $\sigma$ like $O(1/\sqrt{K})$, we get
an $O\left(\frac{\ln(K)}{\sqrt{K}}\right)$ rate. 
\end{theorem}
\end{tcolorbox}

We note that the NGN stepsize~\eqref{eq:stoch-NGN} does not require knowledge of the Lipschitz constant~$L$. 

\begin{tcolorbox}
\begin{theorem}[NGN, strongly convex]
\label{thm:stoc_strong_cvx}
Let $f =\frac{1}{N}\sum_{i=1}^{N} f_i$, where each $f_i:\R^d\to\R$ is non-negative, $L$-smooth and convex. Additionally, assume $f$ is $\mu$-strongly convex. 
Consider the SGD method~\eqref{eq:sgd} with the NGN stepsize~\eqref{eq:stoch-NGN}.
For any value of $\sigma>0$, we have
$$\mathbf{E}\|x^{k+1}-x^*\|^2 \le   {\color{olive}(1-\mu\rho)^k\mathbf{E}\|x^0-x^*\|^2}+ {\color{purple}\frac{6 L}{\mu} \sigma(1+\sigma L)\Delta_{\text{int}}} +{\color{violet}\frac{2\sigma L}{\mu} \max\left\{0,2\sigma L-1\right\} \Delta_{\text{pos}}},
$$
where $\rho =\frac{\sigma}{(1+2\sigma L)(1+\sigma L)}$. Decreasing $\sigma$ like $O(1/K)$, we get
an $O\left(\frac{\ln(K)}{K}\right)$ rate. 
\end{theorem}
\end{tcolorbox}

The proofs of the main results in Thm.~\ref{thm:stoc_cvx} and Thm.~\ref{thm:stoc_strong_cvx} are presented in Section~\ref{sec:convex-proof} and the analysis on decreasing~$\sigma$ is given in Section~\ref{sec:proofs_annealed_sigma}. 
Here, we make a few comments on these results. 

\paragraph{Technical novelty.} Our analysis involves a novel expansion for the expectation of the product of the stochastic NGN stepsize $\gamma_k$ with the stochastic suboptimality value $f_{i_k}(x^k)-f_{i_k}(x^*)$. The correlation between these quantities poses well-known challenges in the analysis of adaptive methods~(see e.g. discussion in~\citet{ward2020adagrad}). Our approach involves decomposing the stepsize as $\gamma_k = \rho+\epsilon_k$, where $\rho = O(\sigma)$ is deterministic and $\epsilon_k = O(\sigma^2)$ is a non-negative stochastic offset. This decomposition allows us to fully benefit from NGN adaptivity in the proof without compromising the bound tightness and the suboptimality level reached at convergence. 

\paragraph{Comment on the rate.} For \textit{any} fixed $\sigma$, the rate is sublinear~(linear in the strongly convex case) to a bounded neighborhood of the solution. The neighborhood magnitude is $O(\sigma)$ -- shrinks as $\sigma$ converges to zero -- and depends on the values of $\Delta_{\text{int}}$ and $\Delta_{\text{pos}}$. Let us elaborate further: 
\begin{itemize}\itemsep 0pt
    \item If $\Delta_{\text{int}}=\Delta_{\text{pos}}=0$, then the results guarantees sublinear~(linear in the strongly convex case) convergence in expectation to the solution for any value of $\sigma$. The best constant in the rate in this setting is achieved when knowing the Lipschitz constant, indeed $\min_{\sigma}\eta_\sigma$ is achieved at $\sigma = 1/(2L)$. 
    
    \item If $\Delta_{\text{pos}}>0$, then the error term $\sigma L \cdot\max\left\{0,2\sigma L-1\right\} \Delta_{\text{pos}}$~(in the convex setting) is non-vanishing\footnote{The threshold $1/(2L)$ is also common in plain SGD, see, e.g., \citet[Theorem 6.8]{garrigos2023handbook}.} for $\sigma>1/(2L)$ -- this does not come from gradient stochasticity but is instead an effect of correcting divergence of the gradient update for large stepsizes. We note that in this case our method behaves more favorably than SGD, which is divergent for large stepsizes. 
    \item If $\Delta_{\text{int}}>0$, then the error term $3\sigma L(1+\sigma L) \Delta_{\text{int}}$~(in the convex setting) is the result of gradient stochasticity around the solution.
    \item By annealing $\sigma$---which does not necessarily yield a monotonically decreasing stepsize $\gamma_k$~(see Imagenet experiment in Figure~\ref{fig:warmup_decay}) -- we get convergence to the solution without knowledge of the gradient Lipshitz constant $L$~(Thm.~\ref{thm:stoc_cvx_dec}).
\end{itemize}

\vspace{-2mm}
\paragraph{Comparison with SPS.} The errors in Thm.~\ref{thm:stoc_cvx} and Thm.~\ref{thm:stoc_strong_cvx} are $O(\sigma)$, meaning that as $\sigma\to0$ the error vanishes and we converge to the solution. This is not the case for recent adaptations of the Polyak scheme to the stochastic setting~\citep{loizou2021stochastic}, where the error is $O(1)$~(Thm.~\ref{thm:loizou}).
\begin{theorem}[Main Theorem of~\citet{loizou2021stochastic}]
   \label{thm:loizou}
Let each $f_i$ be $L_i$-smooth and convex. Denote $L=\max \{L_i\}_{i=1}^n$ the maximum smoothness constant. Consider SGD with 
the SPS$_{\max}$ learning rate $\gamma_k = \min \left\{ \frac{f_{i_k}(x^k)-f_{i_k}^*}{c\|\nabla f_{i_k}(x^k)\|^2}, \gamma_b \right\}$. If $c=1$,  
$\E \left[f(\bar{x}^K)-f(x^*)\right] 
\leq \frac{\|x^0-x^*\|^2}{\alpha \, K} + \frac{2\gamma_{b}\Delta_{\text{int}}}{\alpha},$
where $\alpha=\min \left\{\frac{1}{2cL},\gamma_{b}\right\}$ and $\bar{x}^K=\frac{1}{K}\sum_{k=0}^{K-1} x^k$. If in addition $f$ is $\mu$-strongly convex, then, for any $c\geq1/2$, SGD with SPS$_{\max}$ converges as:
$\E \|x^{k}-x^*\|^2 
\leq \left(1-\mu \alpha \right)^k  \|x^0-x^*\|^2 + \frac{2\gamma_{b} \Delta_{\text{int}} }{\mu \alpha}$.
\end{theorem}
It is easy to realize that indeed, as thoroughly explained by~\citet{orvieto2022dynamics}, the error terms $\frac{2\gamma_{b}\Delta_{\text{int}}}{\alpha}$ and $\frac{2\gamma_{b} \Delta_{\text{int}} }{\mu \alpha}$ in the convex and strongly convex setting are $O(1)$ with respect to the hyperparameters. This is referred to as the bias problem in Section 4 of~\citet{orvieto2022dynamics}. While \citet{orvieto2022vanishing} presented a variation on SPS~(DecSPS) to solve this issue, their proof is based on a classical Adagrad-like argument and additionally needs strong convexity or bounded iterates to guarantee sublinear convergence in the convex domain. Instead, Thm.~\ref{thm:stoc_cvx} requires no assumptions other than convexity and smoothness. Further, in Thm.~\ref{thm:stoc_strong_cvx}, we are able to prove linear convergence to a $O(\sigma)$ neighborhood, a result not yet derived with DecSPS machinery.

\begin{figure}[ht]
    \centering
\includegraphics[width=0.85\textwidth]
{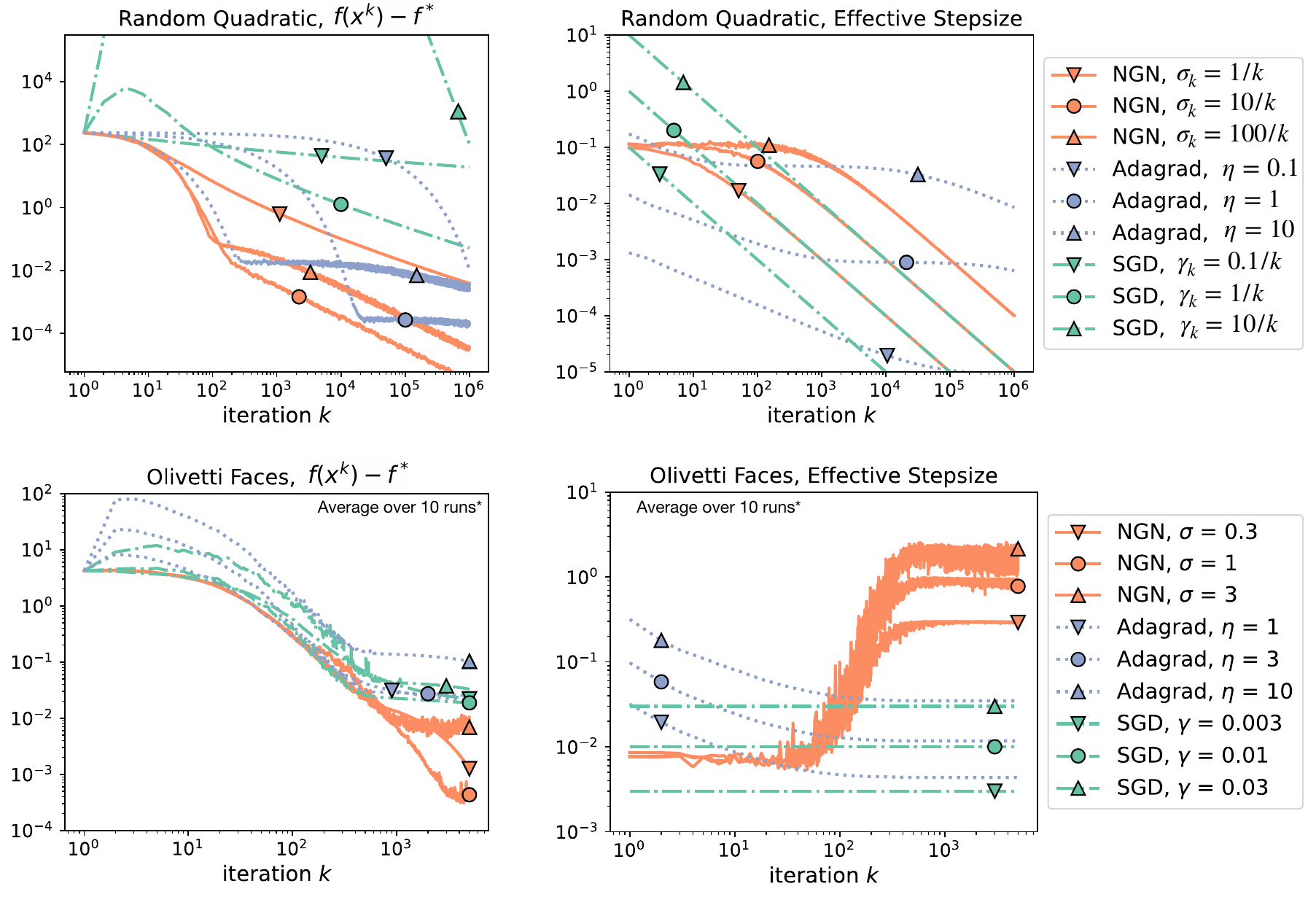}
\vspace{-3mm}
    \caption{Comparison of NGN with Adagrad-norm and SGD for two convex problems, under a decreasing stepsize and a constant stepsize. A comment is provided below.}
\vspace{-3mm}
    \label{fig:quadratic_ada}
\end{figure}

\vspace{-2mm}
\paragraph{Comparison with Adagrad.} Adagrad enjoys state-of-the-art adaptive convergence guarantees in the convex and non-convex settings~\citep{liu2023convergence, faw2022power, wang2023convergence, faw2023beyond}. Knowledge of the gradient Lipschitz constant or additional strong assumptions are not needed for convergence to the problem solution. However, as opposed to NGN, Adagrad's proof techniques do not transfer well to the strongly convex setting~(see instead our Thm.~\ref{thm:stoc_strong_cvx}). In addition, Adagrad is observed to perform suboptimally compared to well-tuned SGD on some problems~\citep{kingma2014adam}. This is mainly due to the fact that the resulting $\gamma_k$ sequence is \textit{forced to be decreasing} by design. A feature that is not present in more commonly used optimizers such as Adam~\citep{kingma2014adam}.

We evaluate the performance of tuned SGD, NGN, and Adagrad-norm in two distinct settings: 
\begin{itemize}\itemsep 0pt
    \item Linear regression on Gaussian input data with random predictor~($d=512$ features, $N=2048$ datapoints: condition number is $\sim 8.5$), optimized with a relatively big batch size of $64$. Here, $\Delta_{\text{pos}}=0$ but $\Delta_{\text{int}}>0$, and we choose decreasing stepsizes for SGD and NGN like $1/k$, to ensure convergence to the unique solution~(see Thm.~\ref{thm:stoc_strong_cvx}). In most recent results \citep[see, e.g.,][]{faw2022power}, Adagrad is not analyzed with a decreasing stepsize $\eta$, so we also do not implement it here.
    \item The second problem is cross-entropy minimization for a linear model on the Olivetti Faces dataset~\citep{pedregosa2011scikit}, where $d=4096$, $N=320$ and we have $40$ distinct classes. Since here $N\ll d$, we include L2 regularization with strength $1e-3$. We optimize with a batch-size of $4$. Realistically, $\Delta_{\text{pos}}, \Delta_{\text{int}}\ll 1$. We, therefore, consider constant hyperparameters in all our optimizers.
\end{itemize}
We report our results in Fig.~\ref{fig:quadratic_ada}. For SGD and NGN, we tuned their single hyperparameter independently to the best performance. For Adagrad-norm, $\gamma_k = \eta/\sqrt{b_0^2+\sum_{j=0}^k \|\nabla f_{i_j}(x_j)\|^2}$ and we only carefully tuned $\eta$ since results are not too sensitive to $b_0$, which we set to $1e-2$. The best-performing hyperparameter for each method is marked with \cc. We also show performance for hyperparameters 3-times-lower~(\td) and 3-times-higher~(\tu). The results clearly indicate, in both settings, that NGN has an edge over Adagrad-norm:
\begin{itemize}
    \item On Olivetti Faces,  where we optimize with constant hyperparameters, we indeed see that Adagrad, even after tuning, does not provide a clear advantage over SGD. Instead, for NGN, performance is remarkably better than SGD and Adagrad for all hyperparameter values. The strong gradients experienced at the first iteration cause the NGN stepsize to automatically warmup, reaching then a $\sigma$-proportional value when gradients become small.
    \item On the linear regression example, the behavior is particularly interesting: NGN prevents the learning rate from overshooting -- unlocking fast progress even at large $\sigma$. Instead, SGD for a big enough learning rate overshoots and only catches up at the end of training. As expected by the result of~\citet{yang2024two}, Adagrad does not overshoot, but its progress slows down as soon as gradients start to vanish~(therefore $\gamma_k$ hardly decreases). Towards the end of training, NGN optimizes with a stepsize $\gamma_k\simeq \sigma_k$, effectively converging to the SGD rule.
\end{itemize}

\begin{remark}[NGN can automatically warmup, but looks like decrease schedule is hand-picked.] While this is indeed the case for the examples seen so far, this property depends on the features of the loss landscape. We show that in deep learning experiments~(Fig.~\ref{fig:warmup_decay}) for a constant $\sigma$ the induced $\gamma_k$ has a warmup-decay behaviour.
    
\end{remark}

\subsubsection{Proofs for main results}
\label{sec:convex-proof}
We present here the proofs for the convex and strongly convex setting.
\label{sec:proofs}
\begin{proof}[\textbf{Proof of Thm.~\ref{thm:stoc_cvx}}]
    By assumption, each $f_{i}$ is convex and $L$-smooth. Expanding the squared distance and using convexity, we get:
\begin{align}
\|x^{k+1}-x^*\|^2 
&=\|(x^{k+1}-x^k) +(x^k-x^*)\|^2 \nonumber \\
&=\|x^k-x^*\|^2+2 \langle x^k-x^*, x^{k+1}-x^k \rangle + \|x^{k+1}-x^k\|^2 \nonumber\\
&=\|x^k-x^*\|^2-2 \gamma_k \langle x^k-x^*, \nabla f_{i_k}(x^k) \rangle + \gamma_k^2 \| \nabla f_{i_k}(x^k)\|^2 \nonumber \\
\label{eq:start_mom}
&\le \|x^k-x^*\|^2-2 \gamma_k [f_{i_k}(x^k)-f_{i_k}(x^*)]+ \gamma_k^2 \| \nabla f_{i_k}(x^k)\|^2.
\end{align}
We now make use of Lemma~\ref{lemma:fundamental_NGN}:
\begin{align}
\|x^{k+1}-x^*\|^2&\le \|x^k-x^*\|^2-2 \gamma_k [f_{i_k}(x^k)-f_{i_k}(x^*)] \nonumber\\
&\quad+ \left(\frac{4\sigma L}{1+2\sigma L}\right)\gamma_k [f_{i_k}(x^k) - f_{i_k}^*] +\frac{2\sigma^2 L}{1+\sigma L} \cdot\max\left\{0,\frac{2\sigma L-1}{2\sigma L+1}\right\} f_{i_k}^*.
\label{eq:cvx_expansion_after_lemma}
\end{align}
Note that taking the expectation conditional on $x^k$ at this point~(as done in the classical SGD proof) is not feasible: indeed, the variable $f_{i_k}(x^k)-f_{i_k}(x^*)$, which would have expectation $f(x^k)-f(x^*)$, is correlated with $\gamma_k$ -- meaning that we would need to consider the expectation of the product $\gamma_k(f_{i_k}(x^k)-f_{i_k}(x^*))$. 

The analysis of \citet{loizou2021stochastic} in the context of SPS consists in writing $f_{i_k}(x^k)-f_{i_k}(x^*) = [f_{i_k}(x^k)-f_{i_k}^*] - [f_{i_k}(x^*)- f_{i_k}^*]$, where both terms within brackets are positive and therefore one can use the stepsize bounds before taking the expectation. This is a smart approach for a quick computation; however it introduces a bias term $\mathbf{E}[\gamma_k(f_{i_k}(x^*)- f_{i_k}^*)]\le\sigma[f(x^*)- \mathbf{E}_{i}f_{i}^*] = O(\sigma)$. It is more desirable, if the method allows, to have error terms only of the order $O(\sigma^2)$, so that one can guarantee later in the proof that, as $\sigma\to 0$, the method converges to the problem solution.

To this end, we write $\gamma_k = \rho + \epsilon_k$, where both $\rho$ and $\epsilon_k$ are non-negative and $\rho$ is deterministic. For intuition, the reader can think of $\rho$ as a stepsize lower bound such that ideally $\rho=O(\sigma)$ and $\epsilon_k=O(\sigma^2)$ for all realizations of $\gamma_k$---we make it more precise in the next lines:
\begin{align*}
\|x^{k+1}-x^*\|^2- \|x^k-x^*\|^2&\le -2 \rho [f_{i_k}(x^k)-f_{i_k}(x^*)] -2\epsilon_k[f_{i_k}(x^k)-f_{i_k}(x^*)]\\
&\quad+ \left(\frac{4\sigma L}{1+2\sigma L}\right)\gamma_k [f_{i_k}(x^k) - f_{i_k}^*] +O(\sigma^2) f_{i_k}^*,
\end{align*}
where we wrote the last term in~\eqref{eq:cvx_expansion_after_lemma} simply as $O(\sigma^2)$ for brevity but will plug in the correct expression at the end of the proof.

At this point, we write $f_{i_k}(x^k)-f_{i_k}(x^*) = [f_{i_k}(x^k)-f_{i_k}^*] - [f_{i_k}(x^*)- f_{i_k}^*]$ \textit{only for the second term in the bound above}. Our purpose, is to make the resulting term $-2\epsilon_k[f_{i_k}(x^k)-f_{i_k}^*]$~(negative) dominant compared to the third term in the bound above~(positive). In formulas, the bound on the distance update $\|x^{k+1}-x^*\|^2- \|x^k-x^*\|^2$ becomes
\begin{equation*}
-2 \rho [f_{i_k}(x^k)-f_{i_k}(x^*)] -2\left(\epsilon_k- \frac{2\sigma L\gamma_k}{1+2\sigma L}\right)[f_{i_k}(x^k)-f_{i_k}^*]+ 2\epsilon_k[f_{i_k}(x^*)-f_{i_k}^*]  +O(\sigma^2) f_{i_k}^*,
\end{equation*}
and we require $\epsilon_k- \frac{2\sigma L}{1+2\sigma L}\gamma_k\ge 0$. Note that $\epsilon_k = \gamma_k-\rho$. Therefore, we need
\begin{equation*}
     \left(1-\frac{2\sigma L}{1+2\sigma L}\right)\gamma_k\ge \rho\quad \implies \gamma_k\ge (1+2\sigma L)\rho.
\end{equation*}
Since $\gamma_k\ge\frac{\sigma}{1+\sigma L}$ thanks to Lemma~\ref{lemma:step_bounds}, we have that a sufficient condition is
\begin{equation*}
    \rho\le\frac{\sigma}{(1+2\sigma L)(1+\sigma L)}.
\end{equation*}
Let us then pick $\rho$ equal to this upper bound. Our bound on the distance update simplifies as
\begin{equation*}
-\frac{2\sigma}{(1+2\sigma L)(1+\sigma L)} [f_{i_k}(x^k)-f_{i_k}(x^*)] + 2\epsilon_k[f_{i_k}(x^*)-f_{i_k}^*]  +O(\sigma^2) f_{i_k}^*.
\end{equation*}
We now get to the interesting part: what is the order of $\epsilon_k$ under our choice for $\rho$?
\begin{equation*}
    \epsilon_k = \gamma_k -\rho\le \sigma -\frac{\sigma}{(1+2\sigma L)(1+\sigma L)} = \frac{\sigma(1+2\sigma L)(1+\sigma L)-\sigma}{(1+2\sigma L)(1+\sigma L)}.
\end{equation*}
Simplifying the bound, we get the desired result: $\epsilon_k = O(\sigma^2)$. Indeed,
\begin{equation}
    \epsilon_k\le \frac{3 L \sigma^2 + 2 L^2 \sigma^3}{(1+2\sigma L)(1+\sigma L)} = L \sigma^2  \frac{3 + 2 L\sigma}{(1+2\sigma L)(1+\sigma L)}\le\frac{3 L \sigma^2}{1+2\sigma L}.
    \label{eq:bound_epsilon}
\end{equation}
All in all, our bound becomes
\begin{equation}
\|x^{k+1}-x^*\|^2- \|x^k-x^*\|^2\le -T_0(\sigma) [f_{i_k}(x^k)-f_{i_k}(x^*)] + T_1(\sigma^2)[f_{i_k}(x^*)-f_{i_k}^*]  +T_2(\sigma^2) f_{i_k}^*,
\label{eq:local_exp_convex}
\end{equation}
where
\begin{align}
\label{eq:T_0}
T_0(\sigma) &= \frac{2\sigma}{(1+2\sigma L)(1+\sigma L)}, \\
\label{eq:T_1}
T_1(\sigma^2) &= \frac{6 L \sigma^2}{1+2\sigma L}, \\
\label{eq:T_2}
T_2(\sigma^2) &=  \frac{2\sigma^2 L}{1+\sigma L} \cdot\max\left\{0,\frac{2\sigma L-1}{2\sigma L+1}\right\}.
\end{align}

We are finally ready to take the conditional expectation with respect to $k$:
\begin{align*}
\mathbf{E}_k\|x^{k+1}-x^*\|^2
&\le \|x^k-x^*\|^2-T_0(\sigma) [f(x^k)-f(x^*)]\\
&\quad +T_1(\sigma^2) \mathbf{E}_k[f_{i_k}(x^*)-f_{i_k}^*] +T_2(\sigma^2) \mathbf{E}_k[f_{i_k}^*].
\end{align*}
Let us call $E_k$ the expected error at step $k$:
\begin{align}
    E_k(\sigma^2)&:=T_1(\sigma^2) \mathbf{E}_k[f_{i_k}(x^*)-f_{i_k}^*] +T_2(\sigma^2) \mathbf{E}_k[f_{i_k}^*]\nonumber\\
    &\le T_1(\sigma^2) \Delta_{\text{int}} +T_2(\sigma^2) \Delta_{\text{pos}},
\label{eq:E_k}
\end{align}
where $\Delta_{\text{int}}:=\mathbf{E}_k[f_{i_k}(x^*)-f_{i_k}^*]$ and $\Delta_{\text{pos}}:=\mathbf{E}_k[f_{i_k}^*]$. Therefore we get the compact formula
\begin{equation*}
    \mathbf{E}_k\|x^{k+1}-x^*\|^2- \|x^k-x^*\|^2\le -T_0(\sigma) [f(x^k)-f(x^*)] + E_k(\sigma^2)
\end{equation*}
Taking the expectation again and using the tower property
\begin{equation*}
    \mathbf{E}\|x^{k+1}-x^*\|^2- \mathbf{E}\|x^k-x^*\|^2\le - T_0(\sigma) \mathbf{E}[f(x^k)-f(x^*)] + E_k(\sigma^2)
\end{equation*}
Finally, averaging over iterations,
\begin{equation*}
    \frac{1}{K}\sum_{k=0}^{K-1}\mathbf{E}\|x^{k+1}-x^*\|^2- \mathbf{E}\|x^k-x^*\|^2\le -\frac{T_0(\sigma)}{K}\sum_{k=0}^{K-1}\mathbf{E}[f(x^k)-f(x^*)] + E_k(\sigma^2).
\end{equation*}
Using linearity of expectation
\begin{equation*}
    \frac{1}{K}\mathbf{E}\|x^{K}-x^*\|^2- \frac{1}{K}\mathbf{E}\|x^0-x^*\|^2\le -T_0(\sigma)\mathbf{E}\left[\sum_{k=0}^{K-1}\frac{1}{K}f(x^k)-f(x^*)\right] + E_k(\sigma^2).
\end{equation*}
and therefore
\begin{equation*}
    \mathbf{E}\left[\frac{1}{ K}\sum_{k=0}^{K-1}f(x^k)-f(x^*)\right] \le \frac{1}{T_0(\sigma) K}\mathbf{E}\|x^0-x^*\|^2  + \frac{E_k(\sigma^2)}{T_0(\sigma)}.
\end{equation*}
Finally, by Jensen's inequality,
\begin{equation*}
    \mathbf{E}\left[f(\bar x^k)-f(x^*)\right] \le \frac{1}{T_0(\sigma) K}\mathbf{E}\|x^0-x^*\|^2  + \frac{E_k(\sigma^2)}{T_0(\sigma)},
\end{equation*}
where $\bar x^K =\frac{1}{K}\sum_{k=0}^{K-1}x^k$.
Finally, using the bound on $E_k(\sigma^2)$ in~\eqref{eq:E_k} and definitions of $T_0$, $T_1$ and $T_2$ in~\eqref{eq:T_0}, \eqref{eq:T_1} and~\eqref{eq:T_2}, respectively, we arrive at
\begin{align*}
    \mathbf{E}\left[f(\bar x^k)-f(x^*)\right] &\le \frac{(1+2\sigma L)^2}{2\sigma K}\mathbf{E}\|x^0-x^*\|^2\\
    &+ 3\sigma L(1+\sigma L) \Delta_{\text{int}} +\sigma L \cdot\max\left\{0,2\sigma L-1\right\} \Delta_{\text{pos}}.
\end{align*}
This concludes the proof.
\end{proof}

\begin{proof}[\textbf{Proof of Theorem~\ref{thm:stoc_strong_cvx}}]

We proceed very similarly with the convex setting
\begin{align*}
\|x^{k+1}-x^*\|^2 
&=\|(x^{k+1}-x^k) +(x^k-x^*)\|^2\\&=\|x^k-x^*\|^2+2 \langle x^k-x^*, x^{k+1}-x^k \rangle + \|x^{k+1}-x^k\|^2\\
&=\|x^k-x^*\|^2-2 \gamma_k \langle x^k-x^*, \nabla f_{i_k}(x^k) \rangle + \gamma_k^2 \| \nabla f_{i_k}(x^k)\|^2
\end{align*}
Making use of Lemma~\ref{lemma:fundamental_NGN}:
\begin{align*}
\|x^{k+1}-x^*\|^2&\le \|x^k-x^*\|^2-2 \gamma_k\langle x^k-x^*, \nabla f_{i_k}(x^k) \rangle\\
&\quad+ \left(\frac{4\sigma L}{1+2\sigma L}\right)\gamma_k [f_{i_k}(x^k) - f_{i_k}^*] +\frac{2\sigma^2 L}{1+\sigma L} \cdot\max\left\{0,\frac{2\sigma L-1}{2\sigma L+1}\right\} f_{i_k}^*.
\end{align*}
As in the proof for Thorem~\ref{thm:stoc_cvx}, we decompose $\gamma$ into a deterministic lower bound $\rho$ and some error $\epsilon_k\in O(\sigma^2)$: $\gamma_k = \rho + \epsilon_k$. As in Thorem~\ref{thm:stoc_cvx} we choose $\rho = \frac{\sigma}{(1+2\sigma L)(1+\sigma L)}$. This leads to (see Equation~\eqref{eq:bound_epsilon})
\begin{equation*}
    \epsilon_k =\gamma_k-\rho \le \frac{3 L \sigma^2}{1+2\sigma L}.
\end{equation*}
Note that we have not yet used convexity. We use it now \textit{after} decomposing $\gamma_k$:
\begin{align*}
&\|x^{k+1}-x^*\|^2\\&\le \|x^k-x^*\|^2-2 \rho\langle x^k-x^*, \nabla f_{i_k}(x^k) \rangle -2 \epsilon_k\langle x^k-x^*, \nabla f_{i_k}(x^k) \rangle\\
&\quad+ \left(\frac{4\sigma L}{1+2\sigma L}\right)\gamma_k [f_{i_k}(x^k) - f_{i_k}^*] +\frac{2\sigma^2 L}{1+\sigma L} \cdot\max\left\{0,\frac{2\sigma L-1}{2\sigma L+1}\right\} f_{i_k}^*\\
&\le \|x^k-x^*\|^2-2 \rho\langle x^k-x^*, \nabla f_{i_k}(x^k) \rangle -2 \epsilon_k[f_{i_k}(x^k)-f_{i_k}(x^*)]\\
&\quad+ \left(\frac{4\sigma L}{1+2\sigma L}\right)\gamma_k [f_{i_k}(x^k) - f_{i_k}^*] +\frac{2\sigma^2 L}{1+\sigma L} \cdot\max\left\{0,\frac{2\sigma L-1}{2\sigma L+1}\right\} f_{i_k}^*.
\end{align*}
Therefore,
\begin{align*}
&\|x^{k+1}-x^*\|^2\\&\le \|x^k-x^*\|^2-2 \rho\langle x^k-x^*, \nabla f_{i_k}(x^k) \rangle -2 \epsilon_k[f_{i_k}(x^k)-f_{i_k}^*] + 2 \epsilon_k[f_{i_k}(x^*)-f_{i_k}^*]\\
&\quad+ \left(\frac{4\sigma L}{1+2\sigma L}\right)\gamma_k [f_{i_k}(x^k) - f_{i_k}^*] +\frac{2\sigma^2 L}{1+\sigma L} \cdot\max\left\{0,\frac{2\sigma L-1}{2\sigma L+1}\right\} f_{i_k}^*\\
&= \|x^k-x^*\|^2-2 \rho\langle x^k-x^*, \nabla f_{i_k}(x^k) \rangle + 2 \epsilon_k[f_{i_k}(x^*)-f_{i_k}^*]\\
&\quad+ 2\left(\epsilon_k-\frac{2\sigma L \gamma_k}{1+2\sigma L}\right)\gamma_k [f_{i_k}(x^k) - f_{i_k}^*] +\frac{2\sigma^2 L}{1+\sigma L} \cdot\max\left\{0,\frac{2\sigma L-1}{2\sigma L+1}\right\} f_{i_k}^*.
\end{align*}
As we know from the proof of for the convex setting, $\epsilon_k-\frac{2\sigma L \gamma_k}{1+2\sigma L}\ge 0$ hence, since $f_{i_k}(x^k) - f_{i_k}^*>0$, we can drop the term. using the upper bound for $\epsilon_k$, we get

\begin{align*}
    &\|x^{k+1}-x^*\|^2 - \|x^k-x^*\|^2\\
    &\le  -2\rho\langle x^k-x^*, \nabla f_{i_k}(x^k) \rangle + \frac{6 L \sigma^2}{1+2\sigma L}[f_{i_k}(x^*)-f_{i_k}^*] + \frac{2\sigma^2 L}{1+\sigma L} \cdot\max\left\{0,\frac{2\sigma L-1}{2\sigma L+1}\right\} f_{i_k}^*.
\end{align*}

Now we take the expectation conditional on $x^k$, recalling our definition for the conditional expectations: $\Delta_{\text{int}}:=\mathbf{E}_k[f_{i_k}(x^*)-f_{i_k}^*]$ and $\Delta_{\text{pos}}:=\mathbf{E}_k[f_{i_k}^*]$.
\begin{align*}
    &\mathbf{E}_k\|x^{k+1}-x^*\|^2 - \|x^k-x^*\|^2\\
    &\le  -2\rho\langle x^k-x^*, \nabla f(x^k) \rangle + \frac{6 L \sigma^2}{1+2\sigma L}\Delta_{\text{int}} + \frac{2\sigma^2 L}{1+\sigma L} \cdot\max\left\{0,\frac{2\sigma L-1}{2\sigma L+1}\right\} \Delta_{\text{pos}}.
\end{align*}
It is now time to use $\mu$-strong convexity of $f$:
\begin{align*}
    \mathbf{E}_k\|x^{k+1}-x^*\|^2 &\le  \|x^k-x^*\|^2 -2\rho\langle x^k-x^*, \nabla f(x^k) \rangle\\
    &+ \frac{6 L \sigma^2}{1+2\sigma L}\Delta_{\text{int}} + \frac{2\sigma^2 L}{1+\sigma L} \cdot\max\left\{0,\frac{2\sigma L-1}{2\sigma L+1}\right\} \Delta_{\text{pos}}\\
    &\le  (1-\mu\rho)\|x^k-x^*\|^2 -2\rho[f(x^k)-f(x^*)]\\
    &+ \frac{6 L \sigma^2}{1+2\sigma L}\Delta_{\text{int}} + \frac{2\sigma^2 L}{1+\sigma L} \cdot\max\left\{0,\frac{2\sigma L-1}{2\sigma L+1}\right\} \Delta_{\text{pos}}.
\end{align*}
Note that $f(x^k)-f(x^*)>0$, since $x^*$ is the minimizer for $f$. Hence, we can drop this term.
\begin{align*}
    \mathbf{E}_k\|x^{k+1}-x^*\|^2 &\le   (1-\mu\rho)\|x^k-x^*\|^2+ \frac{6 L \sigma^2}{1+2\sigma L}\Delta_{\text{int}} + \frac{2\sigma^2 L}{1+\sigma L} \cdot\max\left\{0,\frac{2\sigma L-1}{2\sigma L+1}\right\} \Delta_{\text{pos}}.
\end{align*}
Taking the expectation again and using the tower property,
\begin{align*}
    \mathbf{E}\|x^{k+1}-x^*\|^2 &\le   (1-\mu\rho)\mathbf{E}\|x^k-x^*\|^2+ \frac{6 L \sigma^2}{1+2\sigma L}\Delta_{\text{int}} + \frac{2\sigma^2 L}{1+\sigma L} \cdot\max\left\{0,\frac{2\sigma L-1}{2\sigma L+1}\right\} \Delta_{\text{pos}}.
\end{align*}

This recurrence is of the kind
$y_{k+1} = ay_k + b$, therefore $y_k = a^k y_0 +\left(\sum_{i=0}^{k-1} a^i\right)b$. Since $b\ge0$, we have $y_k \le a^k y_0 + \frac{b}{1-a}$. Therefore,
\begin{align*}
    \mathbf{E}\|x^{k+1}-x^*\|^2 &\le   (1-\mu\rho)^k\mathbf{E}\|x^0-x^*\|^2\\
    &\qquad+ \frac{6 L \sigma^2}{(1+2\sigma L)\rho\mu}\Delta_{\text{int}} + \frac{2\sigma^2 L}{(2\sigma L+1)(1+\sigma L)\rho\mu} \cdot\max\left\{0,2\sigma L-1\right\} \Delta_{\text{pos}}.
\end{align*}
Now recall that $\rho = \frac{\sigma}{(1+2\sigma L)(1+\sigma L)}$. We can then simplify the error term:
\begin{align*}
    \mathbf{E}\|x^{k+1}-x^*\|^2 &\le   (1-\mu\rho)^k\mathbf{E}\|x^0-x^*\|^2\\
    &\qquad+ \frac{6 L \sigma}{\mu}(1+\sigma L)\Delta_{\text{int}} +\frac{2\sigma L}{\mu} \max\left\{0,2\sigma L-1\right\} \Delta_{\text{pos}},
\end{align*}
which is the desired result.
\end{proof}

\subsubsection{General non-convex setting}
\label{sec:nonconvex}

Stochastic NGN can also be applied successfully in the nonconvex setting, as we will see from the deep learning experiments in Section~\ref{sec:experiments}. For convergence analysis in this setting, knowledge of the Lipschitz constant is required and we need to define an additional quantity
$${\color{teal}\Delta^2_{\text{noise}} = \sup_{x\in\R^d} \mathbf{E}[\|\nabla f(x) - \nabla f_{i}(x)\|^2]}, $$
which is the usual bound on stochastic gradient variance. The convergence rate provided here is non-adaptive~(with known~$L$), but we are confident that it can be improved in the future.
\begin{tcolorbox}
\begin{theorem}[NGN, nonconvex]
Let $f =\frac{1}{N}\sum_{i=1}^{N} f_i$, where each $f_i:\R^d\to\R$ is non-negative, $L$-smooth and potentially non-convex. 
Consider the SGD method~\eqref{eq:sgd} with the stochastic NGN stepsize~\eqref{eq:stoch-NGN}.
Then for any $\sigma\le\frac{1}{2L}$, we have
\begin{equation}
\mathbf{E}\left[\frac{1}{K}\sum_{k=0}^{K-1}\|\nabla f(x^k)\|^2\right]\le {\color{olive}\frac{12\cdot \mathbf{E}[f(x^0)-f^*]}{\sigma K}} +{\color{teal}18\sigma L \Delta^2_{\text{noise}}},
\end{equation}
Decreasing $\sigma$ as $O(1/\sqrt{K})$, we get
a convergence rate of $O\left(\frac{\ln(K)}{\sqrt{K}}\right)$.
\label{thm:stoc_nc}
\end{theorem}
\end{tcolorbox}

\begin{proof}
    The proof deviates significantly from the one for stochastic Polyak stepsizes by~\cite{loizou2021stochastic}. We start with the classic expansion based on gradient smoothness
\begin{align*}
    f(x^{k+1})-f(x^k)&\le \langle\nabla f(x^k),x^{k+1}-x^k\rangle + \frac{L}{2}\|x^{k+1}-x^k\|^2\\
    &= -\gamma_k\langle\nabla f(x^k),\nabla f_{i_k}(x^k)\rangle + \frac{L\gamma_k^2}{2}\|\nabla f_{i_k}(x^k)\|^2\\
    &\le -\gamma_k\langle\nabla f(x^k),\nabla f_{i_k}(x^k)\rangle + \frac{L\sigma^2}{2}\|\nabla f_{i_k}(x^k)\|^2.
\end{align*}
We would like to take the conditional expectation with respect to $x^k$. Yet, this is not easy since $\gamma_k$ and $\nabla f_{i_k}(x^k)$ are correlated. Note however that we can write, given the bound in Lemma~\ref{lemma:NGN_lemma},
\begin{equation}
    \gamma_k = \frac{\sigma}{\sigma L + 1} + \frac{\sigma^2 L}{\sigma L +1}\xi_{i_k},
\end{equation}
where $\xi_{i_k}\in[0,1]$ is a random variable. When $\xi_{i_k}=1$ we have $\gamma_k=\sigma$, and when $\xi{i_k}=0$ it holds that $\gamma_k=\frac{\sigma}{\sigma L + 1}$. Thil model for $\gamma_k$ covers its complete range and makes one property explicit : $\gamma_k= O(\sigma)$ with variation range $O(\sigma^2)$. As such, as $\sigma\to 0$ the stepsize becomes deterministic, and the update reduces to SGD with constant stepsize. Leveraging this representation of $\gamma_k$, we can write
\begin{align*}
    -\gamma_k\langle\nabla f(x^k),\nabla f_{i_k}(x^k) &= -\frac{\sigma}{\sigma L +1} \langle\nabla f(x^k),\nabla f_{i_k}(x^k)\rangle-\frac{\sigma^2 L}{\sigma L +1}\xi_{i_k}\langle\nabla f(x^k),\nabla f_{i_k}(x^k)\rangle\\
&\le -\frac{\sigma}{\sigma L +1} \langle\nabla f(x^k),\nabla f_{i_k}(x^k)\rangle+\frac{\sigma^2 L}{\sigma L +1}\left |\xi_{i_k}\right |\cdot \left |\langle\nabla f(x^k),\nabla f_{i_k}(x^k)\rangle\right |\\
&\le -\frac{\sigma}{\sigma L +1} \langle\nabla f(x^k),\nabla f_{i_k}(x^k)\rangle+\frac{\sigma^2 L}{\sigma L +1}\left|\langle\nabla f(x^k),\nabla f_{i_k}(x^k)\rangle\right|.
\end{align*}
Therefore
\begin{equation}\label{eq:inner-prod-bound}
    -\mathbf{E}_{k}[\gamma_k\langle\nabla f(x^k),\nabla f_{i_k}(x^k)\rangle] \le-\frac{\sigma}{\sigma L +1}\|\nabla f(x^k)\|^2+\frac{\sigma^2 L}{\sigma L +1}\mathbf{E}_{k}\left|\langle\nabla f(x^k),\nabla f_{i_k}(x^k)\rangle\right|
\end{equation}
The first term in the above bound is $O(\sigma)$ lnd directly helps convergence, while the last term is an error of $O(\sigma^2)$. Next, recall the basic inequality: for any $a,b\in\R^d$:
\begin{equation*}
    |\langle a,b \rangle|\le \frac{1}{2}\|a\|^2 +\frac{1}{2}\|b\|^2 +\frac{1}{2}\|a-b\|^2.
\end{equation*}
Applying to the last term in~\eqref{eq:inner-prod-bound} and using the assumption $\zeta^2=\sup_{x\in\R^d}\mathbf{E}_i\|\nabla f(x)-\nabla f_i(x)\|^2<\infty$, we have 
\begin{align*}
    2\mathbf{E}_{k}\left|\langle\nabla f(x^k),\nabla f_{i_k}(x^k)\rangle\right|&\le \|\nabla f(x^k)\|^2 +\mathbf{E}_{k}\|\nabla f_{i_k}(x^k)\|^2 +\mathbf{E}_{k}\|\nabla f(x^k)-\nabla f_{i_k}(x^k)\|^2\\
    &\le \|\nabla f(x^k)\|^2 +\mathbf{E}_{k}\|\nabla f_{i_k}(x^k)\|^2 +\zeta^2\\
    &\le 2\|\nabla f(x^k)\|^2 +2\zeta^2.
\end{align*}
Therefore, we get the compact inequality
\begin{align*}
    -\mathbf{E}_{k}[\gamma_k\langle\nabla f(x^k),\nabla f_{i_k}(x^k)\rangle] &\le-\frac{\sigma}{\sigma L +1}\|\nabla f(x^k)\|^2+\frac{\sigma^2 L}{\sigma L +1}\left(\|\nabla f(x^k)\|^2 +\zeta^2\right)\\
    &\le -\sigma\left(\frac{1-\sigma L}{1+\sigma L}\right)\|\nabla f(x^k)\|^2 + \frac{\sigma^2 L}{\sigma L +1} \zeta^2,
\end{align*}
which we can insert back in the original expansion to get
\begin{align*}
    \mathbf{E}_k[f(x^{k+1})]-f(x^k)&\le -\mathbf{E}_k[\gamma_k\langle\nabla f(x^k),\nabla f_{i_k}(x^k)\rangle] + \frac{L\sigma^2}{2}\mathbf{E}_k\|\nabla f_{i_k}(x^k)\|^2\\
    &\le \left[-\sigma\left(\frac{1-\sigma L}{1+\sigma L}\right) + \frac{L\sigma^2}{2}\right]\|\nabla f(x^k)\|^2 + \left(\frac{\sigma^2 L}{\sigma L +1} +\frac{\sigma^2 L}{2}\right)\zeta^2. 
\end{align*}
We therefore need
\begin{equation*}
    -\sigma\left(\frac{1-\sigma L}{1+\sigma L}\right) + \frac{L\sigma^2}{2} = -\sigma\left(\frac{1-\sigma L}{1+\sigma L} -\frac{\sigma L}{2}\right)\le 0.
\end{equation*}
The function $\frac{1-\sigma L}{1+\sigma L} -\frac{\sigma L}{2}$ is monotonically decreasing as $\sigma L>0$ increases. For $\sigma L=0$ it is $1$ and reaches value zero at $-3/2 + \sqrt{17}/2\approxeq 0.56$. For $\sigma L = 0.5$, one gets $\frac{1-\sigma L}{1+\sigma L} -\frac{\sigma L}{2} = \frac{0.5}{1.5}-0.25 = 1/12$. Therefore, for $\sigma\le\frac{1}{2L}$, we get $-\sigma\left(\frac{1-\sigma L}{1+\sigma L}\right) + \frac{L\sigma^2}{2}\le -\frac{\sigma}{12}$.

Next, for the noise term, note that $\frac{\sigma^2 L}{\sigma L +1} +\frac{\sigma^2 L}{2} \le \frac{3\sigma^2 L}{2}$. All in all, for $\sigma\le\frac{1}{2L}$, we get:

\begin{equation*}
    \mathbf{E}_k[f(x^{k+1})]-f(x^k)\le -\frac{\sigma}{12}\|\nabla f(x^k)\|^2 + \frac{3\sigma^2 L}{2}\zeta^2,
\end{equation*}
Or, more conveniently:
\begin{equation*}
    \|\nabla f(x^k)\|^2\le -\frac{12}{\sigma}[\mathbf{E}_k[f(x^{k+1})]-f(x^k)] +18\sigma L \zeta^2.
\end{equation*}
After taking the expectation using the tower property, we get
\begin{equation*}
    \mathbf{E}\|\nabla f(x^k)\|^2\le -\frac{12}{\sigma}[\mathbf{E}[f(x^{k+1})]-\mathbf{E}[f(x^k)]] +18\sigma L \zeta^2,
\end{equation*}
Summing over iterations and telescoping the sum, after adding and subtracting $f^*$ we get
\begin{align*}
    \frac{1}{K}\sum_{k=0}^{K-1}\mathbf{E}\|\nabla f(x^k)\|^2&\le -\frac{12}{\sigma K}\sum_{k=0}^{K-1}\mathbf{E}[f(x^{k+1})]+ \frac{12}{\sigma K}\sum_{k=0}^{K-1}\mathbf{E}[f(x^k)] +18\sigma L \zeta^2\\
&\le -\frac{12}{\sigma K}\mathbf{E}[f(x^{K})]+ \frac{12}{\sigma K}\mathbf{E}[f(x^0)] +18\sigma L \zeta^2\\
&\le -\frac{12}{\sigma K}\mathbf{E}[f(x^{K})-f^*]+ \frac{12}{\sigma K}\mathbf{E}[f(x^0)-f^*] +18\sigma L \zeta^2\\
&\le \frac{12}{\sigma K}\mathbf{E}[f(x^0)-f^*] +18\sigma L \zeta^2.
\end{align*}
This concludes the proof.
\end{proof}

\subsection{Technique sketch for annealed regularization}
\label{sec:proofs_annealed_sigma}

In this section, we analyze stochastic NGN stepsize with decreasing $\sigma_k$, i.e.
\begin{equation}
    \gamma_k = \frac{\sigma_k}{1+\frac{\sigma_k}{2f_{i_k}(x^k)} \ \|\nabla f_{i_k}(x^k)\|^2}.
    \label{eq:NGN_app_dec}
\end{equation}
According to the NGN approximation~\eqref{eq:NGN-approx},
the setting $\sigma_k\to 0$ can be thought of as gradually increasing the regularization strength (i.e. getting more similar to a vanilla SGD update as $k\to\infty$). 

For a general time-dependent $\sigma_k$, all local inequalities hold. In particular, the following three lemmas hold with no additional proof required.

\begin{lemma}[Stepsize bounds, time dependent]
Let each $f_i:\R^d\to\R$ be non-negative, differentiable and $L$-smooth. Consider $\gamma_k$ as in~\eqref{eq:NGN_app_dec}, we have
\begin{equation*}
    \gamma_k \in \left[\frac{\sigma_k}{1+\sigma_k L},\sigma\right] = \left[\frac{1}{L +\sigma_k^{-1}},\sigma_k\right].
\end{equation*}
\label{lemma:step_bounds_decreasing}
\end{lemma}

\begin{lemma}[Fundamental Equality, time dependent]
Consider $\gamma_k$ as in~\eqref{eq:NGN_app_dec}. One has 
\begin{equation*}
    \gamma_k \|\nabla f_{i_k}(x)\|^2= 2\left(\frac{\sigma_k - \gamma_k}{\sigma_k}\right) f_{i_k}(x).
\end{equation*}
\label{lemma:NGN_lemma_decreasing}
\end{lemma}

\begin{lemma}[Fundamental inequality, time dependent]
Let each $f_i:\R^d\to\R$ be non-negative, $L$-smooth and convex. Consider $\gamma_k$ as in~\eqref{eq:NGN_app_dec}, we have
\begin{equation*}
    \gamma_k^2 \| \nabla f_{i_k}(x^k)\|^2\le \left(\frac{4\sigma_k L}{1+2\sigma_k L}\right)\gamma_k [f_{i_k}(x^k) - f_{i_k}^*] + \frac{2\sigma_k^2 L}{1+\sigma_k L} \cdot\max\left\{0,\frac{2\sigma_k L-1}{2\sigma_k L+1}\right\} \cdot f_{i_k}^*.
\end{equation*}
\label{lemma:fundamental_NGN_decreasing}
\end{lemma}

We present proof for the convex setting. The nonconvex case is very similar and uses the same techniques; therefore, we omit it. The strongly convex setting can also be easily derived using e.g. the techniques in~\cite{lacoste2012simpler}.

\begin{theorem}[NGN, convex, decreasing $\sigma_k$]
\label{thm:stoc_cvx_dec}
Let $f =\frac{1}{N}\sum_{i=1}^{N} f_i$, where each $f_i:\R^d\to\R$ is non-negative, $L$-smooth and convex. 
Consider the SGD method~\eqref{eq:sgd} with the stepsize~\eqref{eq:NGN_app_dec}.
For any value of $\sigma_0>0$,
setting $\sigma_k = \sigma_0/\sqrt{k+1}$ leads to the following rate: for $K\ge 2$,
\begin{equation}
    \mathbf{E}[f(\bar x^K)-f(x^*)]\le  \frac{C_1\|x^0-x^*\|^2}{\sqrt{K}-1} + \frac{C_1C_2 \ln(K+1)}{\sqrt{K}-1} = O\left(\frac{\ln(K)}{\sqrt{K}}\right).
\end{equation}
where $\bar x^K = \sum_{k=0}^{K-1} p_{k,K} x^k$ with $p_{k,K}=\frac{\sigma_k}{\sum_{k=0}^{K-1}\sigma_k}$ and
\begin{equation*}
    C_1 = \frac{(1+2\sigma_0 L)(1+\sigma_0 L)}{4\sigma_0},\qquad C_2 = \left[6 \Delta_{\text{int}} + 2 \max\left\{0,2\sigma_0 L-1\right\} \Delta_{\text{pos}}\right] L \sigma_0^2. 
\end{equation*}
\vspace{-2mm}
\end{theorem}
\begin{proof}
Using Lemmas~\ref{lemma:step_bounds_decreasing}, \ref{lemma:NGN_lemma_decreasing} and \ref{lemma:fundamental_NGN_decreasing} and following the same exact steps as in the proof of Theorem~\ref{thm:stoc_cvx}, we arrive at
\begin{equation*}
\|x^{k+1}-x^*\|^2- \|x^k-x^*\|^2\le -T_0(\sigma_k) [f_{i_k}(x^k)-f_{i_k}(x^*)] + T_1(\sigma^2_k)[f_{i_k}(x^*)-f_{i_k}^*]  +T_2(\sigma^2_k) f_{i_k}^*,
\end{equation*}
with 
\begin{align*}
    T_0(\sigma_k) &= \frac{2\sigma_k}{(1+2\sigma_k L)(1+\sigma_k L)},\\
    T_1(\sigma_k^2) &= \frac{6 L \sigma_k^2}{1+2\sigma_k L},\\
    T_2(\sigma_k^2) &=  \frac{2\sigma_k^2 L}{1+\sigma_k L} \cdot\max\left\{0,\frac{2\sigma_k L-1}{2\sigma_k L+1}\right\}
\end{align*}
Taking the expectation, we get the compact formula
\begin{equation*}
    T_0(\sigma_k) \mathbf{E}[f(x^k)-f(x^*)]\le -\mathbf{E}\left[\|x^{k+1}-x^*\|^2 - \|x^k-x^*\|^2\right] + E_k(\sigma^2_k),
\end{equation*}
where $E_k(\sigma_k^2)$ is the expected error at step $k$:
\begin{align}
    E_k(\sigma^2_k)&:=T_1(\sigma^2_k) \mathbf{E}[f_{i_k}(x^*)-f_{i_k}^*] +T_2(\sigma^2_k) \mathbf{E}[f_{i_k}^*]\nonumber\\
    &\le T_1(\sigma^2_k) \Delta_{\text{int}} +T_2(\sigma^2_k) \Delta_{\text{pos}}.
\label{eq:nonconvex-Ek}
\end{align}
Recall the definitions $\Delta_{\text{int}}:=\mathbf{E}[f_{i}(x^*)-f_{i}^*]$ and $\Delta_{\text{pos}}:=\mathbf{E}[f_{i}^*]$. 

Following standard techniques~\citep{garrigos2023handbook}, summing over $k$ and using telescopic cancellation gives
\begin{equation}
    \sum_{k=0}^{K-1}T_0(\sigma_k) \mathbf{E}[f(x^k)-f(x^*)]\le  \|x^0-x^*\|^2 + \sum_{k=0}^{K-1}E_k(\sigma^2_k) ,
\end{equation}
Let us now construct a pointwise lower bound $\tilde T_0(\sigma_k)$ to $T_0(\sigma_k)$, using the fact that $\sigma_k$ is decreasing:
\begin{equation*}
        T_0(\sigma_k) = \frac{2\sigma_k}{(1+2\sigma_k L)(1+\sigma_k L)}\ge \frac{2\sigma_k}{(1+2\sigma_0 L)(1+\sigma_0 L)} =:\tilde T_0(\sigma_k).
\end{equation*}
Then we have
\begin{equation*}
    \sum_{k=0}^{K-1}\tilde T_0(\sigma_k) \mathbf{E}[f(x^k)-f(x^*)]\le  \|x^0-x^*\|^2 + \sum_{k=0}^{K-1}E_k(\sigma^2_k) ,
\end{equation*}

Let us divide every term in the above inequality by $\sum_{k=0}^{K-1}\tilde T_0(\sigma_k)$:
\begin{equation}\label{eq:nonconvex-the-bound}
    \sum_{k=0}^{K-1}\left(\frac{\tilde T_0(\sigma_k)}{\sum_{k=0}^{K-1}\tilde T_0(\sigma_k)}\right) \mathbf{E}[f(x^k)-f(x^*)]\le  \frac{\|x^0-x^*\|^2}{\sum_{k=0}^{K-1}\tilde T_0(\sigma_k)} + \frac{\sum_{k=0}^{K-1}E_k(\sigma^2_k)}{\sum_{k=0}^{K-1}\tilde T_0(\sigma_k)}.
\end{equation}
Using an integral bound, we have, for $K\ge 2$,
\begin{equation*}
    \sum_{k=0}^{K-1}\frac{1}{\sqrt{k+1}}\ge \int_{1}^K\frac{1}{\sqrt{t}}dt = 2(\sqrt{K}-1).
\end{equation*}
Therefore, we have 
\begin{align*}
    \sum_{k=0}^{K-1}\tilde T_0(\sigma_k) &= \sum_{k=0}^{K-1}\frac{2\sigma_k}{(1+2\sigma_0 L)(1+\sigma_0 L)}\\
    &= \frac{2\sigma_0}{(1+2\sigma_0 L)(1+\sigma_0 L)}\sum_{k=0}^{K-1}\frac{1}{\sqrt{k+1}}\\
    &\ge \frac{4\sigma_0(\sqrt{K}-1)}{(1+2\sigma_0 L)(1+\sigma_0 L)}.
\end{align*}
Note also that
\begin{equation*}
    \left(\frac{\tilde T_0(\sigma_k)}{\sum_{k=0}^{K-1}\tilde T_0(\sigma_k)}\right)  = \frac{\sigma_k}{\sum_{k=0}^{K-1}\sigma_k}=: p_{k,K},
\end{equation*}
where $p_{k,K}$ as a function of $k$ is a probability distribution over the interval $[0,K-1]$. 

We are left with bounding $\sum_{k=0}^{K-1}E_k(\sigma_k^2)$. Since $\sigma_k$ is decreasing, we have the following bounds:
\begin{align*}
    T_1(\sigma_k^2) &= \frac{6 L \sigma_k^2}{1+2\sigma_k L}\le 6 L \sigma_k^2,\\
    T_2(\sigma_k^2) &=  \frac{2\sigma_k^2 L}{1+\sigma_k L} \cdot\max\left\{0,\frac{2\sigma_k L-1}{2\sigma_k L+1}\right\}\le 2\sigma_k^2 L \cdot \max\left\{0,2\sigma_0 L-1\right\}.
\end{align*}
Next we use a standard bound on the $K$-th Harmonic number $H_K$.
\begin{equation*}
    \sum_{k=0}^{K-1}\frac{1}{k+1} = 1+ \frac{1}{2} + \frac{1}{3} +\dots + \frac{1}{K} = H_K\le \ln(K+1).
\end{equation*}
Therefore,
\begin{equation*}
    \sum_{k=0}^{K-1}T_1(\sigma_k^2) \le 6 L \sum_{k=0}^{K-1}\sigma_k^2 = 6 L \sigma_0^2\sum_{k=0}^{K-1}\frac{1}{k+1}\le 6 L \sigma_0^2 \ln(K+1),
\end{equation*}
and
\begin{equation*}
    \sum_{k=0}^{K-1}T_2(\sigma_k^2) \le 2 L \cdot \max\left\{0,2\sigma_0 L-1\right\}\sum_{k=0}^{K-1}\sigma_k^2\le2 L \cdot \max\left\{0,2\sigma_0 L-1\right\}\sigma_0^2\ln(K+1).
\end{equation*}
All in all, we substitute the above bounds into~\eqref{eq:nonconvex-Ek} to obtain
\begin{align*}
    \sum_{k=0}^{K-1}E_k(\sigma^2_k)&\le  \Delta_{\text{int}}\sum_{k=0}^{K-1}T_1(\sigma^2_k)  + \Delta_{\text{pos}}\sum_{k=0}^{K-1}T_2(\sigma^2_k) \\
    &\le  \left[6 \Delta_{\text{int}} + 2 \max\left\{0,2\sigma_0 L-1\right\} \Delta_{\text{pos}}\right] L \sigma_0^2 \ln(K+1).
\end{align*}
Plugging everything back into the bound~\eqref{eq:nonconvex-the-bound}, we have, for $K\ge 2$
\begin{equation*}
    \sum_{k=0}^{K-1} p_{k,K} \mathbf{E}[f(x^k)-f(x^*)]\le  \frac{C_1\|x^0-x^*\|^2}{\sqrt{K}-1} + \frac{C_1C_2 \ln(K+1)}{\sqrt{K}-1},
\end{equation*}
with
\begin{equation*}
    C_1 = \frac{(1+2\sigma_0 L)(1+\sigma_0 L)}{4\sigma_0},\qquad C_2 = \left[6 \Delta_{\text{int}} + 2 \max\left\{0,2\sigma_0 L-1\right\} \Delta_{\text{pos}}\right] L \sigma_0^2. 
\end{equation*}
To conclude, let $\bar x^K = \sum_{k=0}^{K-1} p_{k,K} x^k$. Jensen's inequality implies
\begin{equation*}
    f(\bar x^K) = f\left(\sum_{k=0}^{K-1} p_{k,K} x^k\right)\le \sum_{k=0}^{K-1} p_{k,K}f(x^k).
\end{equation*}
This concludes the proof.
\end{proof}

\subsection{Deep learning experiments}
\label{sec:experiments}


In this section, we test the performance of NGN on deep convolutional  networks. We consider the following settings with increasing complexity:
\begin{itemize}
    \item A 12 layers (7 convolutional layers and 5 max-pooling layers, with batch normalization and dropout) neural network trained on the Street View House Numbers (SVHN) Dataset~\citep{netzer2011reading} for 50 epochs with a batch size 512;
    \item A ResNet18~\citep{he2016deep} network\footnote{We use code from the popular repository~\url{https://github.com/kuangliu/pytorch-cifar}.} trained on CIFAR10~\citep{krizhevsky2009learning} for 50 epochs with a batch size 128. 
    \item A ResNet50~\citep{he2016deep} network\footnote{We use the official PyTorch repo \url{https://github.com/pytorch/examples/tree/main/imagenet}.} trained on Imagenet~\citep{deng2009imagenet} for 30 epochs with a batch size 256.
\end{itemize}

All our experiments are performed on NVIDIA V100 GPUs. Reported are mean and confidence interval\footnote{Twice the standard deviation over the square root of number of samples, in our case three.} bars over three seeds for training loss and test accuracy. In Figure~\ref{fig:nn_res} we compare NGN with SGD and Adam~\citep{kingma2014adam}, tuning $\sigma$ in NGN and the learning rate in SGD and Adam. All other parameters in Adam~($\beta_1$ and $\beta_2$) are, as usual in practice, kept constant. While Adam is used with momentum $\beta_1=0.9$, we show here the performance of SGD without momentum to draw a better comparison with our NGN. We also do not make use of L2 regularization as this would bias results on the training speed. All hyperparameters are kept constant during training for SVHN and CIFAR10, while for our largest scale experiment~(Imagenet), we decrease $\sigma$ and learning rates 10-fold every 10 epochs. In Figure~\ref{fig:nn_res2} we show the same NGN statistics as in Figure~\ref{fig:nn_res}, but compare with two adaptive stepsizes with strong guarantees: SPS$_{\text{max}}$ and Adragrad-norm. For SPS: $\gamma_k = \min\left\{\frac{f_{i_k}(x^k)-f_{i_k}^*}{c\|\nabla f(x^k)_{i_k}\|^2},\gamma_b\right\}$ we fix $c=1$ as motivated by the theory in~\cite{loizou2021stochastic} and tune $\gamma_b$. For Adagrad-norm $\gamma_k = \eta/\sqrt{b_0^2+\sum_{j=0}^k \|\nabla f_{i_j}(x_j)\|^2}$ we fix $b_0=1e-2$ and tune $\eta$.

\begin{table}
\centering
{
\small
\begin{tabular}{@{}llccccc@{}}
\toprule
\multirow{2}{*}{Setting} & \multirow{2}{*}{Optimizer} & \multicolumn{5}{c}{\textbf{Hyperparameter}} \\
                         &                            & \#0      & \#1 (\td)    & \#2 (\cc)    & \#3 (\tu)   & \#4     \\ \midrule
\multirow{3}{*}{SVHN + Small Convnet}    & SGD                        & 0.01   & 0.03   &\textbf{ 0.1}   & 0.3  & 1     \\
                         & Adam                       & 0.00003& 0.0001 & \textbf{0.0003}& 0.001& 0.003 \\
                         & Adagrad-norm                        & 1    & 3      & \textbf{10 }    & 30   & 100    \\
                         & SPS                        & 0.3    & 1      & \textbf{3 }    & 10   & 30    \\                   
                         & NGN                        & 0.3    & 1      & \textbf{3 }    & 10   & 30    \\ \midrule
\multirow{3}{*}{CIFAR10 + ResNet18} & SGD                        & 0.03   & 0.1    & \textbf{0.3}   & 1    & 3     \\
                         & Adam                       & 0.00003& 0.0001 & \textbf{0.0003}& 0.001& 0.003 \\
                         & Adagrad-norm                        & 1    & 3      & \textbf{10 }    & 30   & 100    \\
                         & SPS                       &0.1 & 0.3    & \textbf{1 }     & 3    & 10  \\     
                         & NGN                        & 0.3    & 1      & \textbf{3}     & 10   & 30    \\ \midrule
\multirow{3}{*}{Imagenet + ResNet50}& SGD                        & - & 0.3    & \textbf{1 }     & 3     & - \\
                         & Adam                       & - & 0.0003 & \textbf{0.001}  & 0.003 & - \\
                         & Adagrad-norm                      & - & 100 & \textbf{300}  & 1000 & - \\
                         & SPS                       & - & 1 & \textbf{3}  & 10 & - \\
                         & NGN                        & - & 3      & \textbf{10 }    & 30    & - \\ \bottomrule
\end{tabular}}
\caption{Hyperparameters~($\sigma$ for NGN, learning rate for Adam and SGD) by dataset and optimization algorithm, results in Figure~\ref{fig:nn_res}.}
\label{tb:hyper}
\end{table}

As our main objective here is to show how NGN can adapt to the landscape and converge faster than SGD for different values of its hyperparameter, we first grid-search the optimal hyperparameters for all methods on a large logarithmically-spaced grid $[1e-5, 3e-5, 1e-4,\dots, 3, 10, 30]$, and then show results for one/two choices smaller and bigger hyperparameters. Table~\ref{tb:hyper} lists the hyperparameter values used for each method and architecture, where hyperparameter $\# 2$~(in bold) is the best-performing for each method as indicated in the last-train-step train accuracy plots (second column of Figure~\ref{fig:nn_res} and~\ref{fig:nn_res2}).

\begin{figure}[p]
    \centering
\includegraphics[width=0.99\linewidth]{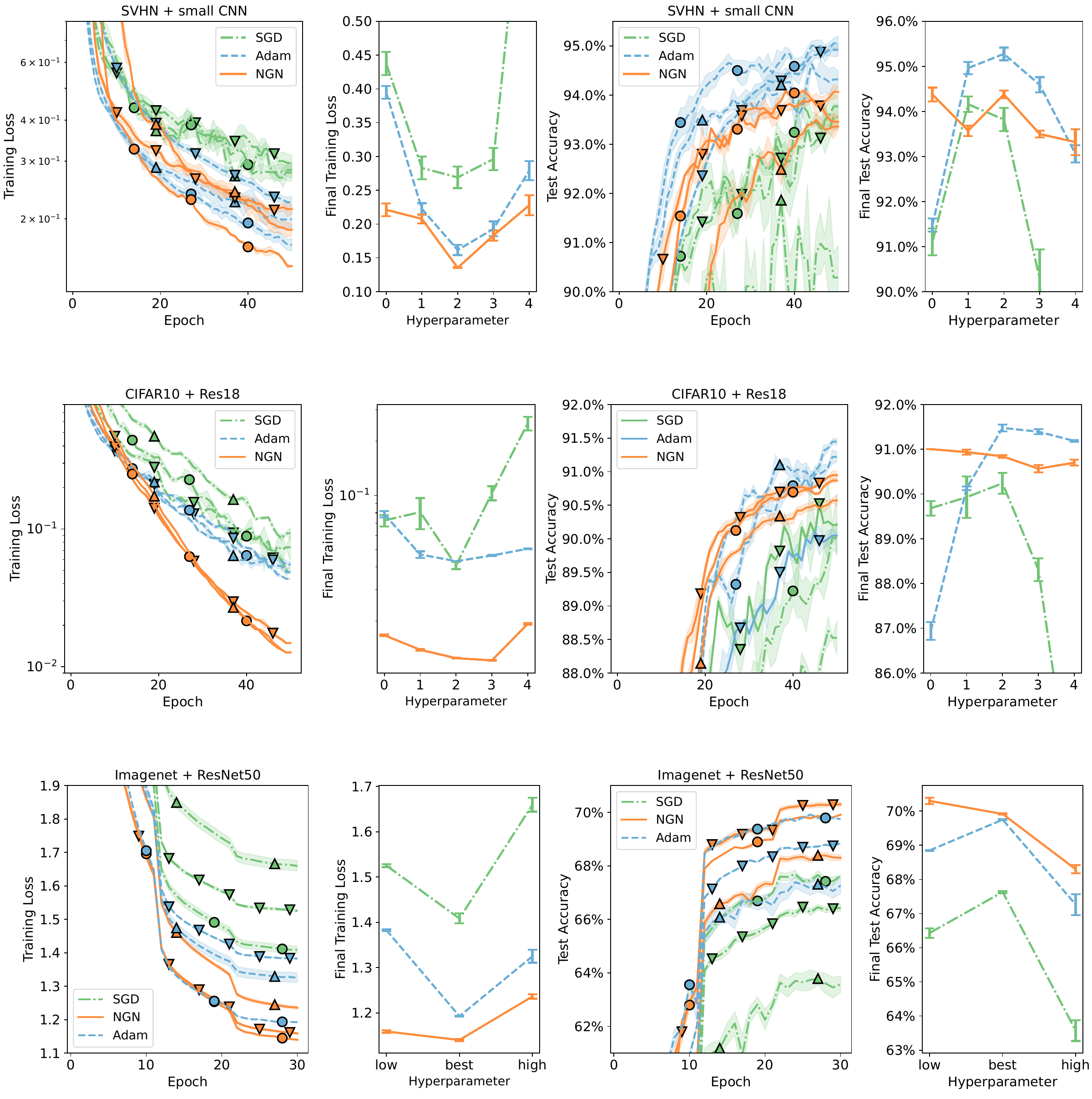}
    \caption{\textbf{(SGD vs. Adam vs. NGN)} Experimental results on Deep Neural Networks~(stochastic gradients). All details and comments can be found in the text. Shown is performance for five or three hyperparameters~(Table~\ref{tb:hyper}), each method is tuned to best at hyperparameter \#2.}
    \label{fig:nn_res}
\end{figure}

\begin{figure}[p]
    \centering
\includegraphics[width=0.99\linewidth]{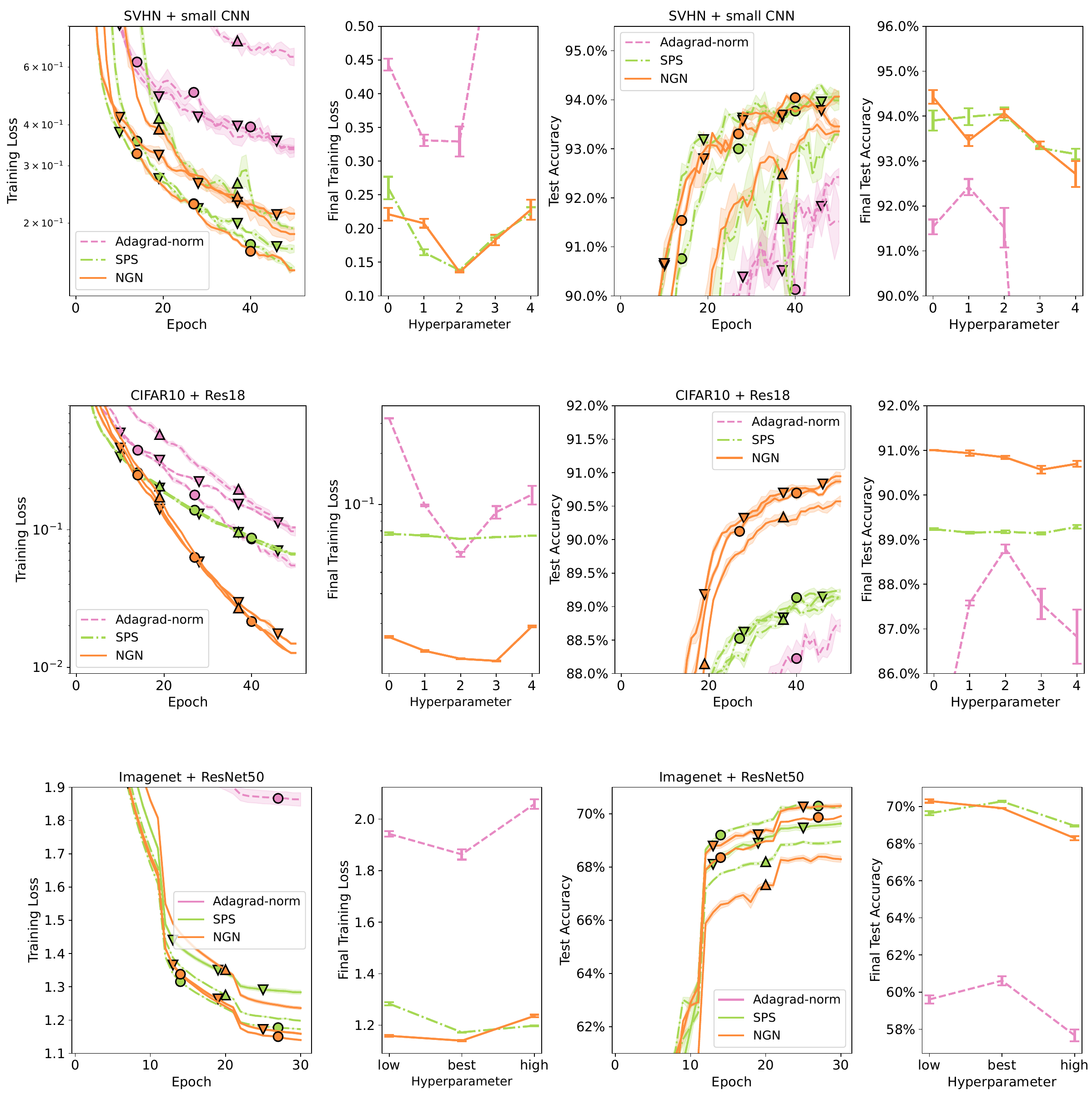}
    \caption{\textbf{(Adagrad-norm vs. SPS vs. NGN)} Experimental results on Deep Neural Networks~(stochastic gradients). All details and comments can be found in the text. Shown is performance for five or three hyperparameters~(Table~\ref{tb:hyper}), each method is tuned to best at hyperparameter \#2.}
    \label{fig:nn_res2}
\end{figure}

While in the second and fourth column of Figure~\ref{fig:nn_res} we show the last-iterate statistics for all hyperparameters in Table~\ref{tb:hyper}, in the first and third column we plot the dynamics for the three best-performing hyperparameters, and we mark them as \td~(hyperparameter $\#1$), \cc~(hyperparameter $\#2$), and \tu~(hyperparameter $\#3$). We now comment on the results in Figure~\ref{fig:nn_res}, i.e. on the comparison with SGD and Adam:
\begin{itemize}
    \item \textit{NGN always performs best in terms of training loss performance}. This holds not only under optimal tuning (hyperparameter $\#2$, \cc) but also across all suboptimal hyperparameters $\#0,\#1, \#3, \#4$, as can be seen in the second column of Figure~\ref{fig:nn_res}: orange curve lower bounds blue and green curves.
    \item \textit{The faster training performance of NGN is more stable across hyperparameters} compared to SGD and Adam. This can be especially seen in ResNets, both on CIFAR10 and Imagenet. Note that the value $\sigma=3$ performs almost optimally for all tasks.
    \item \textit{The test performance of NGN is more robust to tuning} compared to SGD and Adam, as can be seen in the right-most panel in Figure~\ref{fig:nn_res}. However, while on Imagenet, test performance is best for NGN, Adam is superior on SVHN and CIFAR10, with a gap of $0.5-1\%$.
    \item \textit{NGN is more efficient than Adam}: the Adam optimizer maintains second-order statistics for each parameter~\citep{anil2019memory}. This, in large models where GPUs need to be filled to maximum, introduces significant memory overheads that restrict the size of the model being used as well as the number of examples in a mini-batch. NGN, instead, has nearly the same wall-clock and memory complexity as SGD.
\end{itemize}

Although here we did not test the effect of L2 regularization in the interest of avoiding confounders, the promising training speed results makes us believe that NGN, equipped with generalization boosting strategies, such as SAM \citep{foret2020sharpness}, Noise Injection~\citep{orvieto2022anticorrelated, orvieto2023explicit}, or simply decoupled L2 regularization~\citep{loshchilov2017decoupled}, can lead to best generalization performance combined with faster training.

In Figure~\ref{fig:nn_res2} we instead compare NGN with tuned SPS and Adagrad-norm. While we found that Adagrad-norm cannot deliver solid performance, SPS behaves at times similarly to NGN. This is not surprising, since there is a strict relation between SPS and NGN. However, we found that NGN performs drastically better on the ResNet18 setting, with an edge also on Imagenet in train loss.
 
\begin{figure}[t]
    \centering
\includegraphics[width=0.99\linewidth]{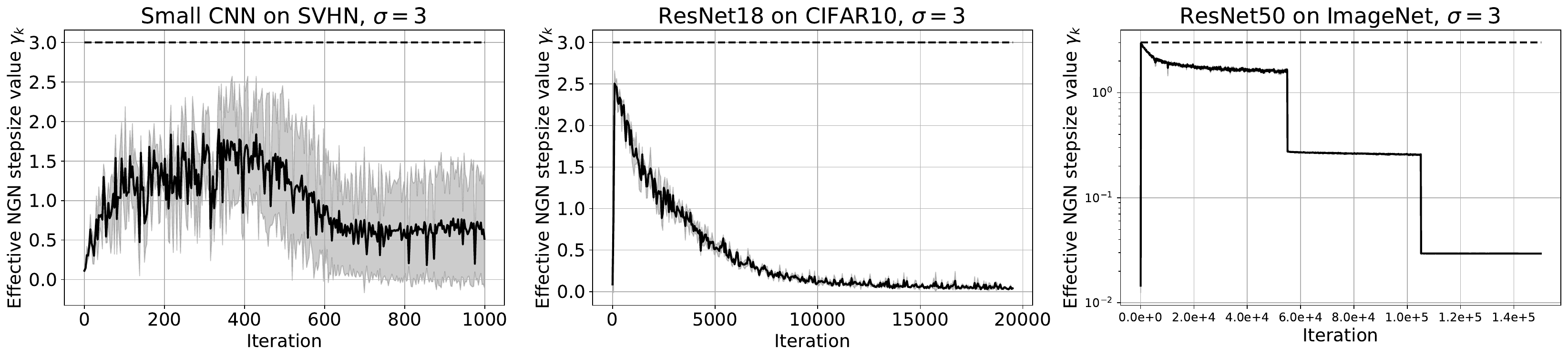}
    \caption{Effective NGN stepsize in the deep learning experiments of Figure~\ref{fig:nn_res}. Shown is mean and 1 standard deviation of $\gamma_k$ for the best-performing $\sigma$.}
    \label{fig:warmup_decay}
\end{figure}

\paragraph{Effective learning rate.} To conclude, we study in Figure~\ref{fig:warmup_decay} the learning rate dynamics in NGN for $\sigma=3$~(uniformly close to optimal performance) in the context of the experiments in this section. On SVHN and CIFAR10, we used vanilla NGN with constant $\sigma$, and found that the stepsize $\gamma_k$ has a peculiar warmup-decay behavior. Interestingly, the duration of the warm-up phase coincides with the usual practice in deep learning~\citep{geiping2023cramming}. Note that while the learning rate decays automatically after warmup, it does not anneal to very small values -- which is needed to counteract variance of stochastic gradients in the large dataset setting. In our Imagenet experiment, supported by the theory in Section~\ref{sec:proofs_annealed_sigma} ,we therefore automatically decay $\sigma$ 10-fold every 30 epochs. The performance of NGN shines in the first 30 epochs exhibiting again a warmup-decay behavior. The learning rate after the first epoch is closer to the one proposed by SGD, but convergence is faster due to the initial highly-adaptive phase.

\section{Conclusion and future work}
\label{sec:conclusion}
In this paper, we introduced NGN, an adaptive optimizer with the same memory requirements and per-iteration cost as SGD. We showed that this algorithm has intriguing and unique curvature estimation properties, leading to features such as stability to hyperparameter tuning and adaptive estimation of the Lipschitz constant.  Our theory supports the empirical findings and the rationale behind the NGN derivation. In the convex setting, NGN is guaranteed to converge with a sublinear rate to any \textit{arbitrarily small neighborhood} of the solution without requiring access to the gradient Lipschitz constant~(as instead needed in SGD) or access to the minimizer value of $f$~(as instead needed in SPS$_{\max}$~\citep{loizou2021stochastic}). Among our results, our rates guarantee adaptive convergence to an arbitrarily small neighborhood with optimal speed in the strongly convex setting, a guarantee that is missing for both Adagrad and Polyak stepsize-based methods. NGN does not require access to the value of the loss at the solution, and by decreasing its hyperparameter~$\sigma$, NGN becomes closer to the SGD update. Our empirical results on neural networks show that NGN is stronger than vanilla SGD, SPS, and Adagrad and can be competitive with Adam even under heavy hyperparameter tuning.

Our work also leaves open many important questions for future research.
For example, it is possible to obtain better convergence rates in the nonconvex case under the Polyak-{\L}ojasiewicz condition \citep[e.g.,][]{Karimi2016PL}.
A very interesting question is how we can incorporate momentum into the NGN update rule and derive improved convergence rates. Along this direction, there are two recent works on using momentum together with the stochastic Polyak stepsize: \cite{wang2023generalized} and \cite{oikonomou2024stochastic}, and we believe that similar extensions for NGN are also possible. 
And finally, it will be very interesting to investigate the (block-) diagonal variant of NGN, where each coordinate (or block coordinates) has its own NGN type of stepsize.   
The coordinate- or block-adaptive variants may help boost performance for training transformers where each parameter group has different curvature property \citep{noci2022signal}.

\bibliographystyle{apalike}
\bibliography{main}

\begin{thebibliography}{}

\bibitem[Anil et~al., 2019]{anil2019memory}
Anil, R., Gupta, V., Koren, T., and Singer, Y. (2019).
\newblock Memory efficient adaptive optimization.
\newblock {\em Advances in Neural Information Processing Systems}, 32.

\bibitem[Berrada et~al., 2020]{berrada2020training}
Berrada, L., Zisserman, A., and Kumar, M.~P. (2020).
\newblock Training neural networks for and by interpolation.
\newblock In {\em International conference on machine learning}, pages 799--809. PMLR.

\bibitem[Bottou et~al., 2018]{bottou2018optimization}
Bottou, L., Curtis, F.~E., and Nocedal, J. (2018).
\newblock Optimization methods for large-scale machine learning.
\newblock {\em SIAM review}, 60(2):223--311.

\bibitem[Chang and Lin, 2011]{chang2011libsvm}
Chang, C.-C. and Lin, C.-J. (2011).
\newblock Libsvm: a library for support vector machines.
\newblock {\em ACM transactions on intelligent systems and technology (TIST)}, 2(3):1--27.

\bibitem[Chen, 2011]{chen2011hessian}
Chen, P. (2011).
\newblock Hessian matrix vs. gauss--newton hessian matrix.
\newblock {\em SIAM Journal on Numerical Analysis}, 49(4):1417--1435.

\bibitem[Cutkosky, 2020]{cutkosky2020free}
Cutkosky, A. (2020).
\newblock Parameter-free, dynamic, and strongly-adaptive online learning.
\newblock In {\em Proceedings of the 37th International Conference on Machine Learning}, volume 119 of {\em Proceedings of Machine Learning Research}, pages 2250--2259. PMLR.

\bibitem[Defazio et~al., 2023]{defazio2023when}
Defazio, A., Cutkosky, A., Mehta, H., and Mishchenko, K. (2023).
\newblock When, why and how much? adaptive learning rate scheduling by refinement.
\newblock arXiv:2310.07831.

\bibitem[Defazio and Mishchenko, 2023]{defazio2023learning}
Defazio, A. and Mishchenko, K. (2023).
\newblock Learning-rate-free learning by d-adaptation.
\newblock In {\em Proceedings of the 40th International Conference on Machine Learning}, volume 202 of {\em Proceedings of Machine Learning Research}, pages 7449--7479. PMLR.

\bibitem[Defazio et~al., 2024]{defazio2024road}
Defazio, A., Xingyu, Yang, Mehta, H., Mishchenko, K., Khaled, A., and Cutkosky, A. (2024).
\newblock The road less scheduled.
\newblock arXiv:2405.15682.

\bibitem[Delyon and Juditsky, 1993]{DelyonJuditsky93}
Delyon, B. and Juditsky, A. (1993).
\newblock Accelerated stochastic approximation.
\newblock {\em SIAM Journal on Optimization}, 3(4):868--881.

\bibitem[Deng et~al., 2009]{deng2009imagenet}
Deng, J., Dong, W., Socher, R., Li, L.-J., Li, K., and Fei-Fei, L. (2009).
\newblock Imagenet: A large-scale hierarchical image database.
\newblock In {\em 2009 IEEE conference on computer vision and pattern recognition}, pages 248--255. Ieee.

\bibitem[Duchi et~al., 2011]{duchi2011adaptive}
Duchi, J., Hazan, E., and Singer, Y. (2011).
\newblock Adaptive subgradient methods for online learning and stochastic optimization.
\newblock {\em Journal of machine learning research}, 12(7).

\bibitem[Faw et~al., 2023]{faw2023beyond}
Faw, M., Rout, L., Caramanis, C., and Shakkottai, S. (2023).
\newblock Beyond uniform smoothness: A stopped analysis of adaptive sgd.
\newblock In {\em The Thirty Sixth Annual Conference on Learning Theory}, pages 89--160. PMLR.

\bibitem[Faw et~al., 2022]{faw2022power}
Faw, M., Tziotis, I., Caramanis, C., Mokhtari, A., Shakkottai, S., and Ward, R. (2022).
\newblock The power of adaptivity in sgd: Self-tuning step sizes with unbounded gradients and affine variance.
\newblock In {\em Conference on Learning Theory}, pages 313--355. PMLR.

\bibitem[Foret et~al., 2020]{foret2020sharpness}
Foret, P., Kleiner, A., Mobahi, H., and Neyshabur, B. (2020).
\newblock Sharpness-aware minimization for efficiently improving generalization.
\newblock {\em arXiv preprint arXiv:2010.01412}.

\bibitem[Garrigos and Gower, 2023]{garrigos2023handbook}
Garrigos, G. and Gower, R.~M. (2023).
\newblock Handbook of convergence theorems for (stochastic) gradient methods.
\newblock {\em arXiv preprint arXiv:2301.11235}.

\bibitem[Geiping and Goldstein, 2023]{geiping2023cramming}
Geiping, J. and Goldstein, T. (2023).
\newblock Cramming: Training a language model on a single gpu in one day.
\newblock In {\em International Conference on Machine Learning}, pages 11117--11143. PMLR.

\bibitem[Ghadimi and Lan, 2013]{ghadimi2013stochastic}
Ghadimi, S. and Lan, G. (2013).
\newblock Stochastic first-and zeroth-order methods for nonconvex stochastic programming.
\newblock {\em SIAM journal on optimization}, 23(4):2341--2368.

\bibitem[Hazan and Kakade, 2019]{hazan2019revisiting}
Hazan, E. and Kakade, S. (2019).
\newblock Revisiting the polyak step size.
\newblock {\em arXiv preprint arXiv:1905.00313}.

\bibitem[He et~al., 2016]{he2016deep}
He, K., Zhang, X., Ren, S., and Sun, J. (2016).
\newblock Deep residual learning for image recognition.
\newblock In {\em Proceedings of the IEEE conference on computer vision and pattern recognition}, pages 770--778.

\bibitem[Hoffmann et~al., 2022]{hoffmann2022training}
Hoffmann, J., Borgeaud, S., Mensch, A., Buchatskaya, E., Cai, T., Rutherford, E., Casas, D. d.~L., Hendricks, L.~A., Welbl, J., Clark, A., et~al. (2022).
\newblock Training compute-optimal large language models.
\newblock {\em arXiv preprint arXiv:2203.15556}.

\bibitem[Jacobs, 1988]{Jacobs88DBD}
Jacobs, R.~A. (1988).
\newblock Increased rates of convergence through learning rate adaption.
\newblock {\em Neural Networks}, 1:295--307.

\bibitem[Karimi et~al., 2016]{Karimi2016PL}
Karimi, H., Nutini, J., and Schmidt, M. (2016).
\newblock Linear convergence of gradient and proximal-gradient methods under the {Polyak}-{\l}ojasiewicz condition.
\newblock In Frasconi, P., Landwehr, N., Manco, G., and Vreeken, J., editors, {\em Machine Learning and Knowledge Discovery in Databases (ECML PKDD 2016)}, volume 9851 of {\em Lectur Notes in Computer Sciencce}. Springer.

\bibitem[Kesten, 1958]{Kesten58}
Kesten, H. (1958).
\newblock Accelerated stochastic approximation.
\newblock {\em Annals of Mathematical Statistics}, 29(1):41--59.

\bibitem[Kingma and Ba, 2014]{kingma2014adam}
Kingma, D.~P. and Ba, J. (2014).
\newblock Adam: A method for stochastic optimization.
\newblock {\em arXiv preprint arXiv:1412.6980}.

\bibitem[Krizhevsky et~al., 2009]{krizhevsky2009learning}
Krizhevsky, A., Hinton, G., et~al. (2009).
\newblock Learning multiple layers of features from tiny images.

\bibitem[Lacoste-Julien et~al., 2012]{lacoste2012simpler}
Lacoste-Julien, S., Schmidt, M., and Bach, F. (2012).
\newblock A simpler approach to obtaining an o (1/t) convergence rate for the projected stochastic subgradient method.
\newblock {\em arXiv preprint arXiv:1212.2002}.

\bibitem[{Le Cun} et~al., 1998]{lecun98efficient}
{Le Cun}, Y., Bottou, L., Orr, G.~B., and M{\"{u}}ller, K.-R. (1998).
\newblock Efficient backprop.
\newblock In {\em Neural Networks, Tricks of the Trade}, Lecture Notes in Computer Science LNCS~1524. Springer Verlag.

\bibitem[Levenberg, 1944]{levenberg44}
Levenberg, K. (1944).
\newblock A method for the solution of certain non-linear problems in least squares.
\newblock {\em Quarterly of Applied Mathematics}, 2(2):164–168.

\bibitem[Liu et~al., 2023]{liu2023convergence}
Liu, Z., Nguyen, T.~D., Ene, A., and Nguyen, H. (2023).
\newblock On the convergence of adagrad (norm) on $\mathbb{R}^{d}$: Beyond convexity, non-asymptotic rate and acceleration.
\newblock In {\em International Conference on Learning Representations}. International Conference on Learning Representations.

\bibitem[Loizou et~al., 2021]{loizou2021stochastic}
Loizou, N., Vaswani, S., Laradji, I.~H., and Lacoste-Julien, S. (2021).
\newblock Stochastic polyak step-size for sgd: An adaptive learning rate for fast convergence.
\newblock In {\em International Conference on Artificial Intelligence and Statistics}, pages 1306--1314. PMLR.

\bibitem[Loshchilov and Hutter, 2017]{loshchilov2017decoupled}
Loshchilov, I. and Hutter, F. (2017).
\newblock Decoupled weight decay regularization.
\newblock {\em arXiv preprint arXiv:1711.05101}.

\bibitem[Ma et~al., 2018]{ma2018power}
Ma, S., Bassily, R., and Belkin, M. (2018).
\newblock The power of interpolation: Understanding the effectiveness of sgd in modern over-parametrized learning.
\newblock In {\em International Conference on Machine Learning}, pages 3325--3334. PMLR.

\bibitem[Mahmood et~al., 2012]{MahmoodSutton12TuningFree}
Mahmood, A.~R., Sutton, R.~S., Degris, T., and Pilarski, P.~M. (2012).
\newblock Tuning-free step-size adaption.
\newblock In {\em Proceedings of the IEEE International Conference on Acoustics, Speech and Signal Processing (ICASSP)}, pages 2121--2124.

\bibitem[Marquardt, 1963]{marquardt63}
Marquardt, D. (1963).
\newblock An algorithm for least-squares estimation of nonlinear parameters.
\newblock {\em SIAM Journal on Applied Mathematics}, 11(2):431–441.

\bibitem[Mirzoakhmedov and Uryasev, 1983]{MirzoakhmedovUryasev83}
Mirzoakhmedov, F. and Uryasev, S.~P. (1983).
\newblock Adaptive step adjustment for a stochastic optimization algorithm.
\newblock {\em Zh. Vychisl. Mat. Mat. Fiz.}, 23(6):1314--1325.
\newblock [U.S.S.R. Comput. Math. Math. Phys. {23:6}, 1983].

\bibitem[Mishchenko and Defazio, 2024]{mishchenko2024prodigy}
Mishchenko, K. and Defazio, A. (2024).
\newblock Prodigy: An expeditiously adaptive parameter-free learner.
\newblock arXiv:2306.06101.

\bibitem[Nesterov, 2018]{nesterov2018lectures}
Nesterov, Y. (2018).
\newblock {\em Lectures on convex optimization}, volume 137 of {\em Springer Optimization and Its Applications}.
\newblock Springer, second edition.

\bibitem[Netzer et~al., 2011]{netzer2011reading}
Netzer, Y., Wang, T., Coates, A., Bissacco, A., Wu, B., Ng, A.~Y., et~al. (2011).
\newblock Reading digits in natural images with unsupervised feature learning.
\newblock In {\em NIPS workshop on deep learning and unsupervised feature learning}. Granada, Spain.

\bibitem[Nocedal and Wright, 2006]{nocedal06book}
Nocedal, J. and Wright, S.~J. (2006).
\newblock {\em Numerical Optimization}.
\newblock Springer Series in Operations Research. Springer, 2nd edition.

\bibitem[Noci et~al., 2022]{noci2022signal}
Noci, L., Anagnostidis, S., Biggio, L., Orvieto, A., Singh, S.~P., and Lucchi, A. (2022).
\newblock Signal propagation in transformers: Theoretical perspectives and the role of rank collapse.

\bibitem[Oikonomou and Loizou, 2024]{oikonomou2024stochastic}
Oikonomou, D. and Loizou, N. (2024).
\newblock Stochastic polyak step-sizes and momentum: Convergence guarantees and practical performance.
\newblock arXiv:2406.04142.

\bibitem[Orabona and P{\'a}l, 2016]{orabona2016coin}
Orabona, F. and P{\'a}l, D. (2016).
\newblock Coin betting and parameter-free online learning.
\newblock {\em Advances in Neural Information Processing Systems}, 29.

\bibitem[Orabona and Tommasi, 2017]{orabona2017training}
Orabona, F. and Tommasi, T. (2017).
\newblock Training deep networks without learning rates through coin betting.
\newblock {\em Advances in Neural Information Processing Systems}, 30.

\bibitem[Orvieto et~al., 2022a]{orvieto2022anticorrelated}
Orvieto, A., Kersting, H., Proske, F., Bach, F., and Lucchi, A. (2022a).
\newblock Anticorrelated noise injection for improved generalization.
\newblock In {\em International Conference on Machine Learning}, pages 17094--17116. PMLR.

\bibitem[Orvieto et~al., 2022b]{orvieto2022vanishing}
Orvieto, A., Kohler, J., Pavllo, D., Hofmann, T., and Lucchi, A. (2022b).
\newblock Vanishing curvature in randomly initialized deep relu networks.
\newblock In {\em International Conference on Artificial Intelligence and Statistics}.

\bibitem[Orvieto et~al., 2022c]{orvieto2022dynamics}
Orvieto, A., Lacoste-Julien, S., and Loizou, N. (2022c).
\newblock Dynamics of {SGD} with stochastic polyak stepsizes: Truly adaptive variants and convergence to exact solution.
\newblock In {\em Advances in Neural Information Processing Systems}.

\bibitem[Orvieto et~al., 2023]{orvieto2023explicit}
Orvieto, A., Raj, A., Kersting, H., and Bach, F. (2023).
\newblock Explicit regularization in overparametrized models via noise injection.
\newblock In {\em International Conference on Artificial Intelligence and Statistics}, pages 7265--7287. PMLR.

\bibitem[Pedregosa et~al., 2011]{pedregosa2011scikit}
Pedregosa, F., Varoquaux, G., Gramfort, A., Michel, V., Thirion, B., Grisel, O., Blondel, M., Prettenhofer, P., Weiss, R., Dubourg, V., et~al. (2011).
\newblock Scikit-learn: Machine learning in python.
\newblock {\em the Journal of machine Learning research}, 12:2825--2830.

\bibitem[Polyak, 1987]{polyak87introduction}
Polyak, B.~T. (1987).
\newblock {\em Introduction to Optimization}.
\newblock Translation series in mathematics and engineering. Optimization Software, Inc., Publications Division, New York, NY.

\bibitem[Prazeres and Oberman, 2021]{prazeres2021stochastic}
Prazeres, M. and Oberman, A.~M. (2021).
\newblock Stochastic gradient descent with polyak’s learning rate.
\newblock {\em Journal of Scientific Computing}, 89:1--16.

\bibitem[Reddi et~al., 2019]{reddi2019convergence}
Reddi, S.~J., Kale, S., and Kumar, S. (2019).
\newblock On the convergence of adam and beyond.
\newblock {\em arXiv preprint arXiv:1904.09237}.

\bibitem[Rolinek and Martius, 2018]{rolinek2018l4}
Rolinek, M. and Martius, G. (2018).
\newblock L4: Practical loss-based stepsize adaptation for deep learning.
\newblock {\em Advances in neural information processing systems}, 31.

\bibitem[Schmidt et~al., 2021]{schmidt2021descending}
Schmidt, R.~M., Schneider, F., and Hennig, P. (2021).
\newblock Descending through a crowded valley-benchmarking deep learning optimizers.
\newblock In {\em International Conference on Machine Learning}, pages 9367--9376. PMLR.

\bibitem[Schraudolph, 1999]{Schraudolph1999}
Schraudolph, N.~N. (1999).
\newblock Local gain adaptation in stochastic gradient descent.
\newblock In {\em Proceedings of Nineth International Conference on Artificial Neural Networks (ICANN)}, pages 569--574.

\bibitem[Sutskever et~al., 2013]{sutskever13momentum}
Sutskever, I., Martens, J., Dahl, G., and Hinton, G. (2013).
\newblock On the importance of initialization and momentum in deep learning.
\newblock In {\em Proceedings of the 30th International Conference on Machine Learning}, volume~28 of {\em Proceedings of Machine Learning Research}, pages 1139--1147, Atlanta, Georgia, USA. PMLR.

\bibitem[Sutton, 1992]{Sutton92IDBD}
Sutton, R.~S. (1992).
\newblock Adapting bias by gradient descent: An incremental version of {Delta-Bar-Delta}.
\newblock In {\em Proceedings of the Tenth National Conference on Artificial Intelligence (AAAI'92)}, pages 171--176. The MIT Press.

\bibitem[Vaswani et~al., 2020]{vaswani2020adaptive}
Vaswani, S., Laradji, I., Kunstner, F., Meng, S.~Y., Schmidt, M., and Lacoste-Julien, S. (2020).
\newblock Adaptive gradient methods converge faster with over-parameterization (but you should do a line-search).
\newblock {\em arXiv preprint arXiv:2006.06835}.

\bibitem[Wang et~al., 2023a]{wang2023convergence}
Wang, B., Zhang, H., Ma, Z., and Chen, W. (2023a).
\newblock Convergence of adagrad for non-convex objectives: Simple proofs and relaxed assumptions.
\newblock In {\em The Thirty Sixth Annual Conference on Learning Theory}, pages 161--190. PMLR.

\bibitem[Wang et~al., 2023b]{wang2023generalized}
Wang, X., Johansson, M., and Zhang, T. (2023b).
\newblock Generalized polyak step size for first order optimization with momentum.
\newblock In {\em Proceedings of the 40th International Conference on Machine Learning}, volume 202 of {\em Proceedings of Machine Learning Research}, pages 35836--35863. PMLR.

\bibitem[Ward et~al., 2020]{ward2020adagrad}
Ward, R., Wu, X., and Bottou, L. (2020).
\newblock Adagrad stepsizes: Sharp convergence over nonconvex landscapes.
\newblock {\em Journal of Machine Learning Research}, 21(219):1--30.

\bibitem[Yang et~al., 2023]{yang2024two}
Yang, J., Li, X., Fatkhullin, I., and He, N. (2023).
\newblock Two sides of one coin: the limits of untuned sgd and the power of adaptive methods.
\newblock {\em Advances in Neural Information Processing Systems}, 36.

\bibitem[Zamani and Glineur, 2023]{zamani2023exact}
Zamani, M. and Glineur, F. (2023).
\newblock Exact convergence rate of the last iterate in subgradient methods.
\newblock arXiv:2307.11134.

\bibitem[Zhang et~al., 2021]{zhang2021understanding}
Zhang, C., Bengio, S., Hardt, M., Recht, B., and Vinyals, O. (2021).
\newblock Understanding deep learning (still) requires rethinking generalization.
\newblock {\em Communications of the ACM}, 64(3):107--115.

\end{thebibliography}

\end{document}